\documentclass[12pt,etds]{amsart}
\usepackage{amsmath}
\usepackage{amssymb}
\usepackage{amsfonts}
\usepackage{amsthm}
\usepackage{enumerate}
\usepackage{tikz}
\usetikzlibrary{matrix}

\usepackage{pb-diagram}

\newtheorem{theorem}{Theorem}[section]
\newtheorem{lemma}[theorem]{Lemma}
\newtheorem{prop}[theorem]{Proposition}
\newtheorem{cor}[theorem]{Corollary}
\theoremstyle{remark}
\newtheorem{definition}[theorem]{Definition}

\def\N{{\mathbb N}}

\def\R{{\mathbb R}}

\def\Z{{\mathbb Z}}

\def\B{{\mathcal{B}}}

\def\K{{\mathcal{K}}}

\def\I{{\mathcal{I}}}

\def\L{{\mathcal{L}}}

\def\M{{\mathcal{M}}}

\newcommand{\clsp}{\overline{\operatorname{span}}}
\newcommand{\lsp}{\operatorname{span}}

\newcommand{\id}{\operatorname{id}}

\newcommand{\piso}{\operatorname{piso}}
\newcommand{\iso}{\operatorname{iso}}

\newcommand{\End}{\operatorname{End}}
\newcommand{\Aut}{\operatorname{Aut}}

\newcommand{\whitesquare}{\hfill $\whitesquare$\newline\vspace{0.4cm}}

\def\newspan{\operatorname{span}}

\numberwithin{equation}{section}

\begin{document}

\title[The partial-isometric crossed products]
{The partial-isometric crossed products by semigroups of endomorphisms are Morita equivalent to crossed products by groups}

\author[Saeid Zahmatkesh]{Saeid Zahmatkesh}
\address{Department of Mathematics, Faculty of Science, King Mongkut's University of Technology Thonburi, Bangkok 10140, THAILAND}
\email{saeid.zk09@gmail.com, saeid.kom@kmutt.ac.th}



\subjclass[2010]{Primary 46L55}
\keywords{$C^*$-algebra, endomorphism, semigroup, partial isometry, crossed product}

\begin{abstract}
Let $\Gamma^{+}$ be the positive cone of a totally ordered abelian discrete group $\Gamma$, and $\alpha$ an action of $\Gamma^{+}$ by extendible endomorphisms of a $C^*$-algebra $A$.
We prove that the partial-isometric crossed product $A\times_{\alpha}^{\piso}\Gamma^{+}$ is a full corner of a group crossed product $\B\times_{\beta}\Gamma$, where $\B$ is a subalgebra of $\ell^{\infty}(\Gamma,A)$ generated by a collection of faithful copies of $A$, and the action $\beta$ on $\B$ is induced by shift on $\ell^{\infty}(\Gamma,A)$.
We then use this realization to show that $A\times_{\alpha}^{\piso}\Gamma^{+}$ has an essential ideal $J$, which is a full corner in an ideal $\I\times_{\beta}\Gamma$ of $\B\times_{\beta}\Gamma$.
\end{abstract}
\maketitle

\section{Introduction}
\label{intro}
Let $\Gamma$ be a totally ordered abelian discrete group. For convenience, we use the additive notation $+$ for the group action. Therefore the identity element is denoted by $0$, and $-x$ denotes the inverse of an element $x\in\Gamma$.
Let $\Gamma^{+}:=\{x\in\Gamma: x\geq 0\}$ be the positive cone of $\Gamma$. Now suppose that $(A,\Gamma^{+},\alpha)$ is a dynamical system consisting of a $C^*$-algebra $A$, an action $\alpha:\Gamma^{+}\rightarrow \End (A)$ of $\Gamma^{+}$ by endomorphisms of $A$ such that
$\alpha_{0}=\id$. Note that since the $C^*$-algebra $A$ may not be unital, we require that each endomorphism $\alpha_{x}$ extends to a strictly continuous endomorphism $\overline{\alpha}_{x}$ of the multiplier algebra $\M(A)$.
This for an endomorphism $\alpha$ of $A$ happens if and only if there exists an approximate identity $(a_{\lambda})$ in $A$ and a projection $p\in \M(A)$ such that $\alpha(a_{\lambda})$ converges strictly to $p$ in $\M(A)$.
But we should note that the extendibility of $\alpha$ does not necessarily imply $\overline{\alpha}(1_{\M(A)})=1_{\M(A)}$. Authors in \cite{LR}, following the idea of isometric crossed products (see \cite{Stacey, ALNR}), for the system $(A,\Gamma^{+},\alpha)$, defined
a covariant representation such that the endomorphisms $\alpha_{x}$ are implemented by partial isometries. Then the associated partial-isometric crossed product $A\times_{\alpha}^{\piso}\Gamma^{+}$ of the system is a $C^*$-algebra generated by a universal covariant representation
such that there is a bijection between covariant representations of the system and nondegenerate representations of $A\times_{\alpha}^{\piso}\Gamma^{+}$. In particular, they analyzed the structure of the partial-isometric crossed product of the distinguished system
$(B_{\Gamma^{+}},\Gamma^{+},\tau)$, where the action $\tau$ of $\Gamma^{+}$ on the subalgebra $B_{\Gamma^{+}}$ of $\ell^{\infty}(\Gamma^{+})$ is given by the right translation. In the line of their work, the present work is actually a generalization of \cite{AZ}.
More precisely, it had previously been known that isometric crossed products are full corners in crossed products by groups (see \cite{Stacey, Adji1}). Then, the authors of \cite{AZ} discussed such an analogous view for partial-isometric crossed products. So
they showed that $A\times_{\alpha}^{\piso}\Gamma^{+}$ is a full corner in a subalgebra of adjointable operators on $\ell^{2}(\Gamma^{+})\otimes A\simeq \ell^{2}(\Gamma^{+},A)$, and
when the action $\alpha$ is given by semigroup of automorphisms, $A\times_{\alpha}^{\piso}\Gamma^{+}$ is a full corner in the group crossed product $(B_{\Gamma}\otimes A)\times\Gamma$. While this attempt successfully leads us to understand more about structure
of partial-isometric crossed products, in particular when $\Gamma^{+}$ is $\N:=\Z^{+}$ (see \cite{AZ2, LZ}), the following question has still been wondering us:

Are partial-isometric crossed products full corners in crossed products by groups in general (similar to isometric crossed products)?

So here in the present work we want to answer this question. Undoubtedly the importance of efforts toward answering such a question lies on the fact that by relating semigroup crossed products to group crossed products, we could enjoy a lot of information
on semigroup crossed products by importing them from the well-established theory of classical crossed products by groups.

By inspiration from \cite{KS}, we define a subalgebra $\B$ of the algebra $\ell^{\infty}(\Gamma,A)$ of norm bounded $A$-valued functions of $\Gamma$. Then the shift on $\ell^{\infty}(\Gamma,A)$ gives an action of $\Gamma$ on $\B$ by automorphisms.
Let $\B\times\Gamma$ be the associated group crossed product of $\B$ by $\Gamma$. Next, we construct a partial-isometric covariant pair $(\pi_{A},W)$ of $(A,\Gamma^{+},\alpha)$ in the multiplier algebra of $\B\times\Gamma$, and we show that
the corresponding homomorphism $\pi_{A}\times W$ of the partial-isometric crossed product is an isomorphism of $A\times_{\alpha}^{\piso}\Gamma^{+}$ onto a full corner of $\B\times\Gamma$. We then apply this realization to identify the ideal
$J$ of $A\times_{\alpha}^{\piso}\Gamma^{+}$ arising from the surjective homomorphism $\phi:A\times_{\alpha}^{\piso}\Gamma^{+}\rightarrow A\times_{\alpha}^{\iso}\Gamma^{+}$ (see \S\ref{sec:pre}) as a full corner in an ideal $\I\times\Gamma$ of $\B\times\Gamma$.
Furthermore, we show that the ideal $J$ is essential, and therefore we generalize the statement of \cite[Propostion 2.5]{AZ} to any totally ordered abelian (discrete) group, not only subgroups of $\R$. Consequently we see that
the main results of \cite{AZ} follow from the present work.

We begin with a preliminary section containing a summary on the theory of partial-isometric and isometric crossed products. In section \ref{sec:alg B} we introduce a subalgebra $\B$ of $\ell^{\infty}(\Gamma,A)$ and its crossed product $\B\times_{\beta}\Gamma$ by $\Gamma$, where
the action $\beta$ is given by the shift on $\ell^{\infty}(\Gamma,A)$. In section \ref{sec:full piso} a partial-isometric covariant pair of $(A,\Gamma^{+},\alpha)$ in $\M(\B\times_{\beta}\Gamma)$ will be constructed, which gives us an isomorphism
of $A\times_{\alpha}^{\piso}\Gamma^{+}$ onto a full corner of $\B\times_{\beta}\Gamma$. From this realization, the kernel of the surjective homomorphism of the partial-isometric crossed onto the isometric crossed product will be identified as a full corner in
an ideal $\I\times_{\beta}\Gamma$ of $\B\times_{\beta}\Gamma$. Moreover, we see that it is an essential ideal. In sections \ref{sec:auto} and \ref{sec:integer} we recover the main results of \cite{AZ} from the present work.

\section{Preliminaries}
\label{sec:pre}

\subsection{Morita equivalence and full corner}
\label{Morita}
What follows is a very quick recall. For more, readers may refer to \cite{RW} or \cite{Lance}. Two $C^*$-algebras $A$ and $B$ are called \emph{Morita equivalent} if there is an $A$--$B$-imprimitivity bimodule $X$. If $p$ is a projection in the multiplier algebra $\M(A)$ of $A$, 
then the  $C^*$-subalgebra $pAp$ of $A$ is called a \emph{corner} in $A$. We
say a corner $pAp$ is \emph{full} if $\overline{ApA}:=\clsp\{apb: a,b\in A\}$ is $A$. Any full corner of $A$ is Morita equivalent to $A$ via the imprimitivity bimodule $Ap$.

\subsection{The partial-isometric crossed product}
\label{piso cp}
A \emph{partial-isometric representation} of $\Gamma^{+}$ on a Hilbert space $H$ is a map $V:\Gamma^{+}\rightarrow B(H)$ such that
$V_{x}:=V(x)$ is a partial isometry, and $V_{x+y}=V_{x}V_{y}$ for all $x,y\in\Gamma^{+}$.
Note that the product $VW$ of two partial isometries $V$ and $W$ is a partial isometry precisely when $V^{*}V$ commutes with $WW^{*}$ \cite[Proposition 2.1]{LR}.
Thus a partial isometry $V$ is called a \emph{power partial isometry} if $V^{n}$ is a partial isometry for every $n\in\N$. So a
partial-isometric  representation of $\N$ is determined by a single power partial isometry $V_{1}$.
If $V$ is a partial-isometric representation of $\Gamma^{+}$, then every $V_{x}V_{x}^{*}$ commutes with $V_{t}V_{t}^{*}$, and so does
$V_{x}^{*}V_{x}$ with $V_{t}^{*}V_{t}$.

A \emph{covariant partial-isometric representation} of $(A,\Gamma^{+},\alpha)$ on a Hilbert space $H$ is a pair $(\pi,V)$ consisting of
a nondegenerate representation $\pi:A\rightarrow B(H)$ and a partial-isometric representation $V:\Gamma^{+}\rightarrow B(H)$ of $\Gamma^{+}$ such that
$\pi(\alpha_{x}(a))=V_{x}\pi(a) V_{x}^{*}$ and $V_{x}^{*}V_{x} \pi(a)=\pi(a) V_{x}^{*}V_{x}$ for all $a\in A$ and $x\in\Gamma^{+}$.

Recall that every system $(A,\Gamma^{+},\alpha)$ admits a nontrivial covariant partial-isometric representation (see \cite[Example 4.6]{LR}). Moreover, by \cite[Lemma 4.2]{LR}, every covariant pair $(\pi,V)$ extends to a partial-isometric covariant representation $(\overline{\pi},V)$ of
$(M(A),\Gamma^{+},\overline{\alpha})$, and the partial-isometric covariance is equivalent to
$\pi(\alpha_{x}(a))V_{x}=V_{x}\pi(a)$ and $V_{x}V_{x}^{*}=\overline{\pi}(\overline{\alpha}_{x}(1))$ for $a\in A$ and $x\in\Gamma^{+}$.

\begin{definition}
\label{piso cp df}
A \emph{partial-isometric crossed product} of $(A,\Gamma^{+},\alpha)$ is a triple $(B,i_{A},i_{\Gamma^{+}})$ consisting of a $C^*$-algebra $B$, a nondegenerate homomorphism $i_{A}:A\rightarrow B$, and a partial-isometric representation $i_{\Gamma^{+}}:\Gamma^{+}\rightarrow \M(B)$ such that:
\begin{itemize}
\item[(i)] the pair $(i_{A},i_{\Gamma^{+}})$ is a covariant representation of $(A,\Gamma^{+},\alpha)$ in $B$;
\item[(ii)] for every covariant partial-isometric representation $(\pi,V)$ of $(A,\Gamma^{+},\alpha)$ on a Hilbert space $H$,
there exists a nondegenerate representation
$\pi\times V: B\rightarrow B(H)$ such that $(\pi\times V) \circ i_{A}=\pi$ and $(\overline{\pi\times V}) \circ i_{\Gamma^{+}}=V$; and
\item[(iii)] $B$ is generated by $i_{A}(A)\cup i_{\Gamma^{+}}(\Gamma^{+})$, we actually have
\[ B=\overline{\newspan}\{i_{\Gamma^{+}}(x)^{*} i_{A}(a) i_{\Gamma^{+}}(y) : x,y \in\Gamma^{+}, a\in A\}. \]
\end{itemize}
By \cite[Proposition 4.7]{LR}, the partial-isometric crossed product of $(A,\Gamma^{+},\alpha)$ always exists, and it is unique up to isomorphism. Thus we write the partial-isometric crossed product $B$ as $A\times_{\alpha}^{\piso}\Gamma^{+}$.
\end{definition}

We recall that by \cite[Theorem 4.8]{LR}, a covariant representation $(\pi, V)$ of $(A,\Gamma^{+},\alpha)$ on $H$ induces a faithful representation $\pi\times V$ of $A\times_{\alpha}^{\piso}\Gamma^{+}$ if and only if $\pi$ is faithful on the range of $(1-V_{s}^{*}V_{s})$ for every $s > 0$.

\subsection{The isometric crossed product}
\label{iso cp}
The isometric crossed product of a system $(A,\Gamma^{+},\alpha)$ is defined similar to the partial-isometric crossed product of the system, such that the endomorphisms $\alpha_{x}$ are implemented by isometries instead of partial-isometries.
We here recall the definition of the isometric crossed product briefly. Readers may refer to \cite{Adji1, ALNR, Stacey} for more details on the theory of isometric crossed products.

An \emph{isometric representation} of $\Gamma^{+}$ on a Hilbert space $H$ is a map $W:\Gamma^{+}\rightarrow B(H)$ such that
$W_{x}:=W(x)$ is an isometry, and $W_{x+y}=W_{x}W_{y}$ for all $x,y\in\Gamma^{+}$.

A \emph{covariant isometric representation} of $(A,\Gamma^{+},\alpha)$ on a Hilbert space $H$ is a pair $(\pi,W)$ consisting of
a nondegenerate representation $\pi:A\rightarrow B(H)$ and an isometric representation of $W:\Gamma^{+}\rightarrow B(H)$ such that
$\pi(\alpha_{x}(a))=W_{x}\pi(a) W_{x}^{*}$ for all $a\in A$ and $x\in\Gamma^{+}$.

Then the \emph{isometric crossed product} $A\times_{\alpha}^{\iso}\Gamma^{+}$ of $(A,\Gamma^{+},\alpha)$ is generated by a universal covariant isometric representation $(k_{A},k_{\Gamma^{+}})$, such that there is a bijection $(\pi,W)\mapsto \pi\times W$
between covariant isometric representations of $(A,\Gamma^{+},\alpha)$ and nondegenerate representations of $A\times_{\alpha}^{\piso}\Gamma^{+}$.
\\

Now, as it is mentioned in \cite[\S 2]{AZ}, we have a nondegenerate surjective homomorphism $\phi:(A\times_{\alpha}^{\piso}\Gamma^{+},i_{A},i_{\Gamma^{+}})\rightarrow (A\times_{\alpha}^{\iso}\Gamma^{+},k_{A},k_{\Gamma^{+}})$ induced by the pair $(k_{A},k_{\Gamma^{+}})$ such that
$$\phi(i_{\Gamma^{+}}(x)^{*} i_{A}(a) i_{\Gamma^{+}}(y))=k_{\Gamma^{+}}(x)^{*} k_{A}(a) k_{\Gamma^{+}}(y),$$ for all $a\in A$ and $x,y\in\Gamma^{+}$. Therefore we have the following short exact sequence:
\begin{align}
\label{exseq1}
0 \longrightarrow \ker \phi \stackrel{}{\longrightarrow} A\times_{\alpha}^{\piso}\Gamma^{+} \stackrel{\phi}{\longrightarrow} A\times_{\alpha}^{\iso}\Gamma^{+} \longrightarrow 0,
\end{align}
where $\ker \phi$ is the ideal
$$J:=\clsp\{i_{\Gamma^{+}}(x)^{*}i_{A}(a)(1-i_{\Gamma^{+}}(t)^{*}i_{\Gamma^{+}}(t))i_{\Gamma^{+}}(y): a\in A, x,y,t\in\Gamma^{+}\}$$ in $A\times_{\alpha}^{\piso}\Gamma^{+}$.
Then in \cite[\S 4]{AZ}, when $\Gamma^{+}=\N$, the ideal $J$ is identified as a full corner of
the algebra $\K(\ell^{2}(\N,A))$ of compact operators on the Hilbert $A$-module $\ell^{2}(\N,A)$. Moreover, when the action $\alpha$ in the system $(A,\Gamma^{+},\alpha)$ is given by semigroups of automorphisms, the ideal $J$ is a full corner of a classical crossed product
by $\Gamma$ (see \cite[\S 5]{AZ}). But the identification of the ideal $J$ in general with familiar terms remained unavailable. So here, first, we show that the partial-isometric crossed product $A\times_{\alpha}^{\piso}\Gamma^{+}$
is a full corner of a classical crossed product by $\Gamma$, from which we see that the ideal $J$ is also a full corner of a group crossed product. Consequently the identification of the ideal $J$ in \cite{AZ} will be recovered here.

Before we proceed to the next section, for Lemma \ref{B extension}, we need to recall that the dynamical system $(A,\Gamma^{+},\alpha)$, in which the action $\alpha$ of $\Gamma^{+}$ is given by (extendible) injective endomorphisms of $A$,
gives a direct system $(A_{s},\varphi_{s}^{t})_{s,t\in\Gamma}$ such that $A_{s}=A$ for every $s\in\Gamma$, and $\varphi_{s}^{t}:A_{s}\rightarrow A_{t}$ is given by $\alpha_{t-s}$ for all $s<t\in\Gamma$.
Let $A_{\infty}$ be the direct limit of the direct system. If $\alpha^{s}:A_{s}\rightarrow A_{\infty}$ is the canonical homomorphism (embedding) of $A_{s}$ into $A_{\infty}$ for every $s\in\Gamma$, then $\cup_{s\in\Gamma}\alpha^{s}(A_{s})$ is a dense subalgebra of $A_{\infty}$.
But note that since $\alpha^{s}(a)=\alpha^{0}(\alpha_{-s}(a))$ for every $s<0$ in $\Gamma$, it follows that $A_{\infty}=\overline{\cup_{s\in\Gamma}\alpha^{s}(A_{s})}=\overline{\cup_{s\in\Gamma^{+}}\alpha^{s}(A_{s})}$.

\section{The $C^*$-algebra $\B$ and its crossed product by $\Gamma$}
\label{sec:alg B}
Let $(A,\Gamma^{+},\alpha)$ be a dynamical system in which $\Gamma^{+}$ is the positive cone of a totally ordered abelian group $\Gamma$, and $\alpha$ is an action of $\Gamma^{+}$ by extendible endomorphisms of a $C^*$-algebra $A$. Let $\ell^{\infty}(\Gamma,A)$ be the $C^*$-algebra of all norm bounded $A$-valued functions of $\Gamma$. For every $x\in\Gamma$, define a map $\mu_{x}:A\rightarrow\ell^{\infty}(\Gamma,A)$ by
\[
(\mu_{x}(a))(y)=
   \begin{cases}
      \alpha_{y-x}(a) &\textrm{if}\empty\ \text{$y\geq x$}\\
      0 &\textrm{if}\empty\ \text{$y<x$}.
   \end{cases}
\]
One can see that each map $\mu_{x}$ is an injective $*$-homomorphism (isometry). Then define $\B$ to be the $C^*$-subalgebra of $\ell^{\infty}(\Gamma,A)$ generated by $\{\mu_{x}(a):x\in\Gamma, a\in A\}$. Since by some simple calculations,
we have $\mu_{x}(a)\mu_{y}(b)=\mu_{z}(\alpha_{z-x}(a)\alpha_{z-y}(b))$, where $z=\textrm{max}\{x,y\}$, and $\mu_{x}(a)^{*}=\mu_{x}(a^{*})$, it follows that $$\B=\clsp\{\mu_{x}(a):x\in\Gamma, a\in A\}.$$

Moreover note that the elements of $\B$ satisfy the following property:
\begin{lemma}
\label{B property}
Let $\xi\in\B$. For any $\varepsilon>0$, there are $y,z\in\Gamma$ such that if $x<y$, then $\|\xi(x)\|<\varepsilon$, and if $x\geq z$, then
$\|\xi(x)-\alpha_{x-z}(\xi(z))\|<\varepsilon$.
\end{lemma}

\begin{proof}
For any $\epsilon>0$, there is a finite sum $\sum_{i=0}^{n}\mu_{y_{i}}(a_{i})$ such that
\begin{align}
\label{eq9}
\|\xi-\sum_{i=0}^{n}\mu_{y_{i}}(a_{i})\|<\varepsilon/2.
\end{align}
Take $y=\textrm{min}\{y_{0},...,y_{n}\}$. Then for all $x\in\Gamma$, we have
$$\|\big(\xi-\sum_{i=0}^{n}\mu_{y_{i}}(a_{i})\big)(x)\|\leq\|\xi-\sum_{i=0}^{n}\mu_{y_{i}}(a_{i})\|<\varepsilon/2<\varepsilon,$$ and therefore $\|\big(\xi-\sum_{i=0}^{n}\mu_{y_{i}}(a_{i})\big)(x)\|<\varepsilon$.
So, as $$\|\big(\xi-\sum_{i=0}^{n}\mu_{y_{i}}(a_{i})\big)(x)\|=\|\xi(x)-\sum_{i=0}^{n}(\mu_{y_{i}}(a_{i}))(x)\|,$$ it follows that
$$\|\xi(x)-\sum_{i=0}^{n}(\mu_{y_{i}}(a_{i}))(x)\|<\varepsilon.$$
Now, if $x<y$, then $(\mu_{y_{i}}(a_{i}))(x)=0$ for each $0\leq i\leq n$, because $x<y_{i}$. Thus we obtain $\|\xi(x)\|<\varepsilon$ for every $x<y$.

Next, take $z=\textrm{max}\{y_{0},...,y_{n}\}$, and for convenience, let $\xi_{n}=\sum_{i=0}^{n}\mu_{y_{i}}(a_{i})$. Since $\|\xi-\xi_{n}\|<\varepsilon/2$ by (\ref{eq9}), for every $x\geq z$, we get
\begin{eqnarray*}
\begin{array}{rcl}
\|\xi(x)-\alpha_{x-z}(\xi(z))\|&=&\|(\xi-\xi_{n}+\xi_{n})(x)-\alpha_{x-z}\big((\xi-\xi_{n}+\xi_{n})(z)\big)\|\\
&=&\|(\xi-\xi_{n})(x)+\xi_{n}(x)-\alpha_{x-z}\big((\xi-\xi_{n})(z)+\xi_{n}(z)\big)\|\\
&=&\|(\xi-\xi_{n})(x)+\xi_{n}(x)-\alpha_{x-z}\big((\xi-\xi_{n})(z)\big)-\alpha_{x-z}\big(\xi_{n}(z)\big)\|\\
&\leq&\|(\xi-\xi_{n})(x)\|+\|-\alpha_{x-z}\big((\xi-\xi_{n})(z)\big)\|+\|\xi_{n}(x)-\alpha_{x-z}\big(\xi_{n}(z)\big)\|\\
&\leq&\|\xi-\xi_{n}\|+\|(\xi-\xi_{n})(z)\|+\|\xi_{n}(x)-\alpha_{x-z}\big(\xi_{n}(z)\big)\|\\
&\leq&\|\xi-\xi_{n}\|+\|\xi-\xi_{n}\|+\|\xi_{n}(x)-\alpha_{x-z}\big(\xi_{n}(z)\big)\|\\
&<&\varepsilon/2+\varepsilon/2+\|\xi_{n}(x)-\alpha_{x-z}\big(\xi_{n}(z)\big)\|=\varepsilon+\|\xi_{n}(x)-\alpha_{x-z}\big(\xi_{n}(z)\big)\|.\\
\end{array}
\end{eqnarray*}
Therefore $$\|\xi(x)-\alpha_{x-z}(\xi(z))\|<\varepsilon+\|\xi_{n}(x)-\alpha_{x-z}\big(\xi_{n}(z)\big)\|.$$ But since $x\geq z$, we actually have $\|\xi_{n}(x)-\alpha_{x-z}\big(\xi_{n}(z)\big)\|=0$ on the right hand side of the inequality. This is because
\begin{eqnarray*}
\begin{array}{rcl}
\xi_{n}(x)-\alpha_{x-z}\big(\xi_{n}(z)\big)&=&\sum_{i=0}^{n}(\mu_{y_{i}}(a_{i}))(x)-\alpha_{x-z}\big(\sum_{i=0}^{n}(\mu_{y_{i}}(a_{i}))(z)\big)\\
&=&\sum_{i=0}^{n}\alpha_{x-y_{i}}(a_{i})-\alpha_{x-z}\big(\sum_{i=0}^{n}\alpha_{z-y_{i}}(a_{i})\big)\\
&=&\sum_{i=0}^{n}\alpha_{x-y_{i}}(a_{i})-\sum_{i=0}^{n}\alpha_{x-z}\big(\alpha_{z-y_{i}}(a_{i})\big)\\
&=&\sum_{i=0}^{n}\alpha_{x-y_{i}}(a_{i})-\sum_{i=0}^{n}\alpha_{x-y_{i}}(a_{i})=0.
\end{array}
\end{eqnarray*}
So it follows that $\|\xi(x)-\alpha_{x-z}(\xi(z))\|<\varepsilon$ for every $x\geq z$.
\end{proof}

\begin{lemma}
\label{mux}
Each homomorphism $\mu_{x}:A\rightarrow\mathcal{B}$ is extendible, meaning that it extends to a strictly continuous homomorphism $\overline{\mu_{x}}:\M(A)\rightarrow\M(\B)$ of multiplier algebras.
\end{lemma}

\begin{proof}
For proof, note that the extendibility of each $\mu_{x}$ essentially follows from the extendibility of each endomorphism $\alpha_{x}$. So, let  $\{a_{i}\}$ be an approximate identity in $A$. We show that there is a projection $p_{x}\in\M(\B)$
such that $\mu_{x}(a_{i})\rightarrow p_{x}$ in $\M(\B)$ strictly.
To do this, it is enough to see that $\mu_{x}(a_{i})\mu_{y}(a)\rightarrow p_{x}\mu_{y}(a)$ and $\mu_{y}(a)\mu_{x}(a_{i})\rightarrow \mu_{y}(a)p_{x}$ in the norm of $\B$ for every $a\in A$ and $y\in \Gamma$. Take the algebra $\ell^{\infty}(\Gamma,\M(A))$ which contains $\ell^{\infty}(\Gamma,A)$ as an essential ideal. Then similar to $\B$, let $\overline{\B}$ be the $C^*$-subalgebra of $\ell^{\infty}(\Gamma,\M(A))$ spanned by $\{\gamma_{r}(m):r\in\Gamma, m\in \M(A)\}$, where each $\gamma_{r}:\M(A)\rightarrow\ell^{\infty}(\Gamma,\M(A))$ is a map defined by
\[
(\gamma_{r}(m))(s)=
   \begin{cases}
      \overline{\alpha}_{s-r}(m) &\textrm{if}\empty\ \text{$s\geq r$}\\
      0 &\textrm{if}\empty\ \text{$s<r$}.
   \end{cases}
\]
So, $\gamma_{r}$ is an isometry for every $r\in\Gamma$, such that $\gamma_{r}|_{A}=\mu_{r}$. Now since $\B$ sits in $\overline{\B}$ as an essential ideal, $\overline{\B}$ is embedded in $\M(\B)$ as a $C^*$-subalgebra. Let $p_{x}=\gamma_{x}(1)$ for every $x\in\Gamma$, which is a projection in $\M(\B)$. If $x\leq y$, then $\mu_{x}(a_{i})\mu_{y}(a)=\mu_{y}(\alpha_{y-x}(a_{i})a)$ and $\mu_{y}(a)\mu_{x}(a_{i})=\mu_{y}(a\alpha_{y-x}(a_{i}))$, which converge to $\mu_{y}(\overline{\alpha}_{y-x}(1)a)$ and $\mu_{y}(a\overline{\alpha}_{y-x}(1))$ in the norm of $\B$ respectively.
This is because $\alpha_{y-x}$ is extendible and $\mu_{y}$ is an isometry. But $\mu_{y}(\overline{\alpha}_{y-x}(1)a)=p_{x}\mu_{y}(a)$ and $\mu_{y}(a\overline{\alpha}_{y-x}(1))=\mu_{y}(a)p_{x}$, and therefore $\mu_{x}(a_{i})\mu_{y}(a)\rightarrow p_{x}\mu_{y}(a)$ and $\mu_{y}(a)\mu_{x}(a_{i})\rightarrow \mu_{y}(a)p_{x}$ when $x\leq y$. A similar discussion shows that these also hold when $x>y$. Thus each $\mu_{x}$ is extendible.
\end{proof}
Note that it then follows from Lemma \ref{mux} that $\overline{\mu_{x}}=\gamma_{x}$ for every $x\in\Gamma$.
\\
Now let $\I$ be the $C^*$-subalgebra of $\B$ generated by $\{\mu_{x}(a)-\mu_{y}(\alpha_{y-x}(a)):x<y\in\Gamma, a\in A\}$. We actually have
$$\I=\clsp\{\mu_{x}(a)-\mu_{y}(\alpha_{y-x}(a)):x<y\in\Gamma, a\in A\}.$$ Moreover, by some computations on spanning elements, one can see that $\I$ is in fact an ideal of $\B$. We claim that $\I$ is an essential ideal. This is because if $\xi\I=0$ for some $\xi\in\B$, then for each $x\in\Gamma$, $$\xi[\mu_{x}(\xi(x)^{*})-\mu_{y}(\alpha_{y-x}(\xi(x)^{*}))]=0,$$ where $y\in\Gamma$ with $y>x$.
So we must have $\xi(x)\xi(x)^{*}=0$ in $A$ for each $x\in\Gamma$, which implies that $\xi(x)=0$. Thus $\xi=0$, and therefore $\I$ is essential.

By a simple calculation, we also have
\begin{align}
\label{eq3}
[\mu_{x}(a)-\mu_{y}(\alpha_{y-x}(a))][\mu_{x}(b)-\mu_{y}(\alpha_{y-x}(b))]=\mu_{x}(ab)-\mu_{y}(\alpha_{y-x}(ab)),
\end{align}
for all $a,b\in A$ and $x<y\in\Gamma$. This equation will be used later in Lemma \ref{ker}.
\begin{lemma}
\label{B extension}
Let $(A,\Gamma^{+},\alpha)$ be a dynamical system in which the action $\alpha$ of $\Gamma^{+}$ is given by (extendible) injective endomorphisms of $A$. Then there is a surjective homomorphism $\sigma:\B\rightarrow A_{\infty}$ such that $\sigma(\mu_{y}(a))=\alpha^{y}(a)$ for every $a\in A$ and $y\in\Gamma$. Moreover,
\begin{align}
\label{ker sigma}
\ker \sigma=\{\xi\in\B:\ \textrm{the net}\ \{\|\xi(x)\|\}_{x\in\Gamma}\ \textrm{converges to}\ 0\},
\end{align}
which contains the ideal $\I$.
\end{lemma}

\begin{proof}
Define a map of the dense subalgebra $\lsp\{\mu_{y}(a):y\in\Gamma, a\in A\}$ of $\B$ into $A_{\infty}$ such that $\sum_{i=0}^{n}\mu_{y_{i}}(a_{i})\mapsto \sum_{i=0}^{n}\alpha^{y_{i}}(a_{i})$. This map is well-defined, because if $z=\textrm{max}\{y_{0},...,y_{n}\}$, then
\begin{eqnarray*}
\begin{array}{rcl}
\|\sum_{i=0}^{n}\alpha^{y_{i}}(a_{i})\|&=&\|\sum_{i=0}^{n}\alpha^{z}(\alpha_{z-y_{i}}(a_{i}))\|\\
&=&\|\alpha^{z}\big(\sum_{i=0}^{n}\alpha_{z-y_{i}}(a_{i})\big)\|\\
&=&\|\sum_{i=0}^{n}\alpha_{z-y_{i}}(a_{i})\|\\
&=&\|\sum_{i=0}^{n}(\mu_{y_{i}}(a_{i}))(z)\|\\
&=&\|\big(\sum_{i=0}^{n}\mu_{y_{i}}(a_{i})\big)(z)\|\\
&\leq&\|\sum_{i=0}^{n}\mu_{y_{i}}(a_{i})\|.
\end{array}
\end{eqnarray*}

Therefore we have $$\|\sum_{i=0}^{n}\alpha^{y_{i}}(a_{i})\|\leq\|\sum_{i=0}^{n}\mu_{y_{i}}(a_{i})\|,$$ which implies that the map $\sum_{i=0}^{n}\mu_{y_{i}}(a_{i})\mapsto \sum_{i=0}^{n}\alpha^{y_{i}}(a_{i})$ is well-defined. It is also linear and bounded, and hence it extends to
a bounded linear map $\sigma:\B\rightarrow A_{\infty}$ of $\B$ into $A_{\infty}$ such that $\sigma(\mu_{y}(a))=\alpha^{y}(a)$. Moreover,
$$\sigma(\mu_{y}(a))^{*}=\alpha^{y}(a)^{*}=\alpha^{y}(a^{*})=\sigma(\mu_{y}(a^{*}))=\sigma(\mu_{y}(a)^{*}),$$ and
$$\sigma(\mu_{x}(a)\mu_{y}(b))=\sigma\big(\mu_{z}(\alpha_{z-x}(a)\alpha_{z-y}(b))\big)=\alpha^{z}(\alpha_{z-x}(a)\alpha_{z-y}(b))=\alpha^{x}(a)\alpha^{y}(b),$$ where $z=\textrm{max}\{x,y\}$, imply that $\sigma$ is indeed a $*$-homomorphism.
The surjectivity of $\sigma$ cane be easily seen from $\sigma(\mu_{y}(a))=\alpha^{y}(a)$ for all $a\in A$ and $y\in\Gamma$.

Now, before we identify $\ker \sigma$, it is not difficult to see that $\ker \sigma$ contains the ideal $\I$. This is because for any spanning element $\mu_{x}(a)-\mu_{y}(\alpha_{y-x}(a))$ of $\I$, where $a\in A$ and $x<y\in\Gamma$, we have
$$\sigma(\mu_{x}(a)-\mu_{y}(\alpha_{y-x}(a)))=\alpha^{x}(a)-\alpha^{y}(\alpha_{y-x}(a))=\alpha^{y}(\alpha_{y-x}(a))-\alpha^{y}(\alpha_{y-x}(a))=0.$$

To prove (\ref{ker sigma}), first, let $\xi\in \ker \sigma$. For every $\varepsilon>0$, there is a finite sum $\sum_{i=0}^{n}\mu_{y_{i}}(a_{i})$ such that $\|\sum_{i=0}^{n}\mu_{y_{i}}(a_{i})-\xi\|<\varepsilon/2$. Then
\begin{eqnarray}
\label{eq6}
\begin{array}{rcl}
\|\sum_{i=0}^{n}\alpha^{y_{i}}(a_{i})\|&=&\|\sigma(\sum_{i=0}^{n}\mu_{y_{i}}(a_{i}))-\sigma(\xi)\|\\
&=&\|\sigma(\sum_{i=0}^{n}\mu_{y_{i}}(a_{i})-\xi)\|\\
&\leq&\|\sum_{i=0}^{n}\mu_{y_{i}}(a_{i})-\xi\|<\varepsilon/2.
\end{array}
\end{eqnarray}
Therefore it follows that, if $z=\textrm{max}\{y_{0},...,y_{n}\}$, then for every $x\geq z$, since
\begin{eqnarray}
\label{eq12}
\begin{array}{rcl}
\|\sum_{i=0}^{n}\alpha_{x-y_{i}}(a_{i})\|&=&\|\sum_{i=0}^{n}\alpha_{x-z}\big(\alpha_{z-y_{i}}(a_{i})\big)\|\\
&=&\|\alpha_{x-z}\big(\sum_{i=0}^{n}\alpha_{z-y_{i}}(a_{i})\big)\|\\
&=&\|\sum_{i=0}^{n}\alpha_{z-y_{i}}(a_{i})\|\\
&=&\|\alpha^{z}\big(\sum_{i=0}^{n}\alpha_{z-y_{i}}(a_{i})\big)\|\\
&=&\|\sum_{i=0}^{n}\alpha^{z}(\alpha_{z-y_{i}}(a_{i}))\|\\
&=&\|\sum_{i=0}^{n}\alpha^{y_{i}}(a_{i})\|,
\end{array}
\end{eqnarray}
we have
 \begin{align}
\label{eq7}
\|\sum_{i=0}^{n}\alpha_{x-y_{i}}(a_{i})\|=\|\sum_{i=0}^{n}\alpha^{y_{i}}(a_{i})\|\leq \|\sum_{i=0}^{n}\mu_{y_{i}}(a_{i})-\xi\|<\varepsilon/2.
\end{align}

Consequently, if $x\geq z$, then by (\ref{eq7})
\begin{eqnarray*}
\begin{array}{rcl}
\|\xi(x)\|&=&\|\xi(x)-\big(\sum_{i=0}^{n}\mu_{y_{i}}(a_{i})\big)(x)+\big(\sum_{i=0}^{n}\mu_{y_{i}}(a_{i})\big)(x)\|\\
&=&\|\big(\xi-\sum_{i=0}^{n}\mu_{y_{i}}(a_{i})\big)(x)+\sum_{i=0}^{n}\big(\mu_{y_{i}}(a_{i})\big)(x)\|\\
&\leq&\|\big(\xi-\sum_{i=0}^{n}\mu_{y_{i}}(a_{i})\big)(x)\|+\|\sum_{i=0}^{n}\alpha_{x-y_{i}}(a_{i})\|\\
&\leq&\|\xi-\sum_{i=0}^{n}\mu_{y_{i}}(a_{i})\|+\|\sum_{i=0}^{n}\alpha_{x-y_{i}}(a_{i})\|<\varepsilon/2+\varepsilon/2=\varepsilon.
\end{array}
\end{eqnarray*}
This shows that the net $\{\|\xi(x)\|\}_{x\in\Gamma}$ converges to $0$. So $\ker \sigma$ is contained in the right hand side of (\ref{ker sigma}). To see the other inclusion,
let $\xi$ belong to the right hand side of (\ref{ker sigma}). We show that $\sigma(\xi)=0$. For every $\varepsilon>0$, there exists $s\in\Gamma$ such that
$\|\xi(x)\|<\varepsilon/3$ for every $x\geq s$, and a finite sum $\sum_{i=0}^{n}\mu_{y_{i}}(a_{i})$ such that $\|\xi-\sum_{i=0}^{n}\mu_{y_{i}}(a_{i})\|<\varepsilon/3$. Let $z=\textrm{max}\{s,y_{0},...,y_{n}\}$. Then we have
\begin{eqnarray}
\label{eq8}
\begin{array}{rcl}
\|\sigma(\xi)\|&=&\|\sigma(\xi-\sum_{i=0}^{n}\mu_{y_{i}}(a_{i}))+\sigma(\sum_{i=0}^{n}\mu_{y_{i}}(a_{i}))\|\\
&\leq&\|\sigma(\xi-\sum_{i=0}^{n}\mu_{y_{i}}(a_{i}))\|+\|\sigma(\sum_{i=0}^{n}\mu_{y_{i}}(a_{i}))\|\\
&\leq&\|\xi-\sum_{i=0}^{n}\mu_{y_{i}}(a_{i})\|+\|\sum_{i=0}^{n}\alpha^{y_{i}}(a_{i})\|\\
&<&\varepsilon/3+\|\sum_{i=0}^{n}\alpha^{y_{i}}(a_{i})\|.
\end{array}
\end{eqnarray}
But in the bottom line, by the same computation as (\ref{eq12}), we have $\|\sum_{i=0}^{n}\alpha^{y_{i}}(a_{i})\|=\|\sum_{i=0}^{n}\alpha_{x-y_{i}}(a_{i})\|$ for every $x\geq z$. So it follows that
$$\|\sigma(\xi)\|<\varepsilon/3+\|\sum_{i=0}^{n}\alpha_{x-y_{i}}(a_{i})\|\ \ \textrm{for all}\ x\geq z.$$
Moreover, if $x\geq z$, then
\begin{eqnarray*}
\begin{array}{rcl}
\|\sum_{i=0}^{n}\alpha_{x-y_{i}}(a_{i})\|&=&\|\sum_{i=0}^{n}\big(\mu_{y_{i}}(a_{i})\big)(x)-\xi(x)+\xi(x)\|\\
&=&\|\big(\sum_{i=0}^{n}\mu_{y_{i}}(a_{i})\big)(x)-\xi(x)+\xi(x)\|\\
&=&\|\big(\sum_{i=0}^{n}\mu_{y_{i}}(a_{i})-\xi\big)(x)+\xi(x)\|\\
&\leq&\|\big(\sum_{i=0}^{n}\mu_{y_{i}}(a_{i})-\xi\big)(x)\|+\|\xi(x)\|\\
&\leq&\|\sum_{i=0}^{n}\mu_{y_{i}}(a_{i})-\xi\|+\|\xi(x)\|<\varepsilon/3+\varepsilon/3=2\varepsilon/3.\\
\end{array}
\end{eqnarray*}
Thus, if $x\geq z$, then
$$\|\sigma(\xi)\|<\varepsilon/3+\|\sum_{i=0}^{n}\alpha_{x-y_{i}}(a_{i})\|\leq\varepsilon/3+2\varepsilon/3=\varepsilon,$$
which means $\|\sigma(\xi)\|<\varepsilon$ for every $\varepsilon>0$. So we must have $\|\sigma(\xi)\|=0$, and therefore $\sigma(\xi)=0$. This completes the proof.
\end{proof}

There is an action $\beta$ of $\Gamma$ by automorphisms on $\B$ induced by the shift on $\ell^{\infty}(\Gamma,A)$, such that $\beta_{s}\circ\mu_{x}=\mu_{x+s}$ for all $s,x\in\Gamma$. Thus we obtain a dynamical system $(\B,\Gamma,\beta)$. Define a map $\rho:\B\rightarrow \L(\ell^{2}(\Gamma,A))$ by $(\rho(\xi)f)(x)=\xi(x)f(x)$, and $U:\Gamma\rightarrow \L(\ell^{2}(\Gamma,A))$ by $(U_{y}f)(x)=f(x-y)$, where $\xi\in\B$ and $f\in\ell^{2}(\Gamma,A)$. Then $\rho$
is a nondegenerate representation, and $U$ is a unitary representation such that we have $\rho(\beta_{s}(\xi))=U_{s}\rho(\xi)U_{s}^{*}$. Therefore the pair $(\rho,U)$ is a covariant representation of $(\B,\Gamma,\beta)$ on $\ell^{2}(\Gamma,A)$.
Moreover, let $(\B\times_{\beta}\Gamma,j_{\B},j_{\Gamma})$ be the group crossed product associated to the system $(\B,\Gamma,\beta)$. Since $\I$ is a $\beta$-invariant essential ideal of $\B$, $\I\times_{\beta}\Gamma$ sits in $\B\times_{\beta}\Gamma$ as an essential ideal \cite[Proposition 2.4]{Kusuda}.

\section{The partial-isometric crossed product $A\times_{\alpha}^{\piso}\Gamma^{+}$ as a full corner of $\B\times_{\beta}\Gamma$}
\label{sec:full piso}

\begin{theorem}
\label{main}
Suppose that $(A,\Gamma^{+},\alpha)$ is a dynamical system consisting of a $C^{*}$-algebra $A$ and an action $\alpha$ of $\Gamma^{+}$ by extendible endomorphisms of $A$. Let $q=\overline{j_{\B}\circ\mu_{0}}(1)$, and
$$\pi_{A}:A\rightarrow q(\B\times_{\beta}\Gamma)q\ \ \textrm{and}\ \ W:\Gamma^{+}\rightarrow \mathcal{M}(q(\B\times_{\beta}\Gamma)q)$$
be the maps defined by $\pi_{A}(a)=(j_{\B}\circ\mu_{0})(a)$ and $W_{x}=qj_{\Gamma}(x)^{*}q$ for all $a\in A$ and $x\in\Gamma^{+}$. Then the pair $(\pi_{A},W)$ is a partial-isometric covariant representation of $(A,\Gamma^{+},\alpha)$, and
the associated homomorphism $\Psi:=\pi_{A}\times W$ is an isomorphism of $(A\times_{\alpha}^{\piso}\Gamma^{+},i_{A},v)$ onto $q(\B\times_{\beta}\Gamma)q$, such that $\Psi(i_{A}(a))=\pi_{A}(a)$ and $\overline{\Psi}(v_{x})=W_{x}$. Moreover, $A\times_{\alpha}^{\piso}\Gamma^{+}$ is a full corner in $\B\times_{\beta}\Gamma$.
\end{theorem}

\begin{proof}
Firstly $\pi_{A}$ is nondegenerate. This is because for an approximate identity $\{a_{i}\}$ in $A$, since $\mu_{0}(a_{i})\rightarrow\overline{\mu_{0}}(1)$ in $\M(\B)$ strictly, and $j_{\B}$ is nondegenerate, we have
$\pi_{A}(a_{i})\rightarrow\overline{j_{\B}}(\overline{\mu_{0}}(1))$ in $\mathcal{M}(\B\times_{\beta}\Gamma)$ strictly, where $\overline{j_{\B}}(\overline{\mu_{0}}(1))=\overline{j_{\B}\circ\mu_{0}}(1)=q$. Next we show that $W$ is a partial-isometric representation.
For each $x\in\Gamma^{+}$, by using the covariance equation of the pair $(j_{\B},j_{\Gamma})$, we have
\begin{eqnarray}\label{eq10}
\begin{array}{rcl}
W_{x}W_{x}^{*}W_{x}&=&qj_{\Gamma}(x)^{*}qj_{\Gamma}(x)\overline{j_{\B}}(\overline{\mu_{0}}(1))j_{\Gamma}(x)^{*}q\\
&=&qj_{\Gamma}(x)^{*}q\overline{j_{\B}}(\overline{\beta}_{x}(\overline{\mu_{0}}(1)))q\\
&=&qj_{\Gamma}(x)^{*}q\overline{j_{\B}}(\overline{\mu_{x}}(1))q\\
&=&qj_{\Gamma}(x)^{*}\overline{j_{\B}}(\overline{\mu_{0}}(1))\overline{j_{\B}}(\overline{\mu_{x}}(1))q\\
&=&qj_{\Gamma}(x)^{*}\overline{j_{\B}}(\overline{\mu_{0}}(1)\overline{\mu_{x}}(1))q=qj_{\Gamma}(x)^{*}\overline{j_{\B}}(\overline{\mu_{x}}(\overline{\alpha}_{x}(1)))q.\\
\end{array}
\end{eqnarray}
To continue, for the last term, again we apply the covariance equation of $(j_{\B},j_{\Gamma})$ to obtain
\begin{eqnarray}\label{eq11}
\begin{array}{rcl}
W_{x}W_{x}^{*}W_{x}&=&q\overline{j_{\B}}(\overline{\mu_{0}}(\overline{\alpha}_{x}(1)))j_{\Gamma}(x)^{*}q\\
&=&q\overline{j_{\B}}(\overline{\mu_{0}}(1)\overline{\mu_{-x}}(1))j_{\Gamma}(x)^{*}q\ \ \big[\overline{\mu_{0}}(\overline{\alpha}_{x}(1))=\overline{\mu_{0}}(1)\overline{\mu_{-x}}(1)\big]\\
&=&q\overline{j_{\B}}(\overline{\mu_{0}}(1))\overline{j_{\B}}(\overline{\mu_{-x}}(1))j_{\Gamma}(x)^{*}q\\
&=&q\overline{j_{\B}}(\overline{\mu_{-x}}(1))j_{\Gamma}(x)^{*}q\\
&=&qj_{\Gamma}(x)^{*}\overline{j_{\B}}(\overline{\beta}_{x}(\overline{\mu_{-x}}(1)))q\\
&=&qj_{\Gamma}(x)^{*}\overline{j_{\B}}(\overline{\mu_{0}}(1))q=qj_{\Gamma}(x)^{*}q=W_{x}.\\
\end{array}
\end{eqnarray}
So each $W_{x}$ is a partial-isometry. Also for every $x,y\in\Gamma^{+}$, again by the covariance equation of $(j_{\B},j_{\Gamma})$,
\begin{eqnarray*}
\begin{array}{rcl}
W_{x}W_{y}&=&qj_{\Gamma}(x)^{*}\overline{j_{\B}}(\overline{\mu_{0}}(1))j_{\Gamma}(y)^{*}q\\
&=&qj_{\Gamma}(x)^{*}j_{\Gamma}(y)^{*}\overline{j_{\B}}(\overline{\beta}_{y}(\overline{\mu_{0}}(1)))q\\
&=&qj_{\Gamma}(x+y)^{*}\overline{j_{\B}}(\overline{\mu_{y}}(1))q\\
&=&q\overline{j_{\B}}(\overline{\mu_{0}}(1))j_{\Gamma}(x+y)^{*}\overline{j_{\B}}(\overline{\mu_{y}}(1))q\\
&=&qj_{\Gamma}(x+y)^{*}\overline{j_{\B}}(\overline{\beta}_{x+y}(\overline{\mu_{0}}(1)))\overline{j_{\B}}(\overline{\mu_{y}}(1))q\\
&=&qj_{\Gamma}(x+y)^{*}\overline{j_{\B}}(\overline{\mu_{x+y}}(1))\overline{j_{\B}}(\overline{\mu_{y}}(1))q\\
&=&qj_{\Gamma}(x+y)^{*}\overline{j_{\B}}(\overline{\mu_{x+y}}(1)\overline{\mu_{y}}(1))q\\
&=&qj_{\Gamma}(x+y)^{*}\overline{j_{\B}}(\overline{\mu_{x+y}}(\overline{\alpha}_{x}(1)))\overline{j_{\B}}(\overline{\mu_{0}}(1))q\ \ \big[\overline{\mu_{x+y}}(1)\overline{\mu_{y}}(1)=\overline{\mu_{x+y}}(\overline{\alpha}_{x}(1))\big]\\
&=&qj_{\Gamma}(x+y)^{*}\overline{j_{\B}}(\overline{\mu_{x+y}}(\overline{\alpha}_{x}(1))\overline{\mu_{0}}(1))q\\
&=&qj_{\Gamma}(x+y)^{*}\overline{j_{\B}}(\overline{\mu_{x+y}}(\overline{\alpha}_{x}(1)\overline{\alpha}_{x+y}(1)))q\ \ \big[\overline{\mu_{x+y}}(\overline{\alpha}_{x}(1))\overline{\mu_{0}}(1)=\overline{\mu_{x+y}}(\overline{\alpha}_{x}(1))\overline{\mu_{x+y}}(\overline{\alpha}_{x+y}(1))\big]\\
&=&qj_{\Gamma}(x+y)^{*}\overline{j_{\B}}(\overline{\mu_{x+y}}(\overline{\alpha}_{x+y}(1)))q.\ \ \big[\overline{\alpha}_{x}(1)\overline{\alpha}_{x+y}(1)=\overline{\alpha}_{x}(1\overline{\alpha}_{y}(1))=\overline{\alpha}_{x+y}(1)\big]\\
\end{array}
\end{eqnarray*}
Then for the bottom line, by a similar computation to (\ref{eq11}) (regarding the proof of $W_{x}W_{x}^{*}W_{x}=W_{x}$), where we showed $qj_{\Gamma}(x)^{*}\overline{j_{\B}}(\overline{\mu_{x}}(\overline{\alpha}_{x}(1)))q=W_{x}$, we get
$$W_{x}W_{y}=qj_{\Gamma}(x+y)^{*}\overline{j_{\B}}(\overline{\mu_{x+y}}(\overline{\alpha}_{x+y}(1)))q=W_{x+y}.$$
Thus $W$ is a partial-isometric representation of $\Gamma^{+}$ in $\mathcal{M}(q(\B\times_{\beta}\Gamma)q)\simeq q\mathcal{M}(\B\times_{\beta}\Gamma)q$.

Now we show that the pair $(\pi_{A},W)$ satisfies the partial-isometric covariance equations. For every $a\in A$ and $x\in\Gamma^{+}$,
\begin{eqnarray*}
\begin{array}{rcl}
W_{x}\pi_{A}(a)W_{x}^{*}&=&qj_{\Gamma}(x)^{*}j_{\B}(\mu_{0}(a))j_{\Gamma}(x)q\\
&=&qj_{\B}(\beta_{-x}(\mu_{0}(a)))q\ \ [\textrm{by the covariance equation of}\ (j_{\B},j_{\Gamma})]\\
&=&qj_{\B}(\mu_{-x}(a))q\\
&=&q\overline{j_{\B}}(\overline{\mu_{0}}(1))j_{\B}(\mu_{-x}(a))q\\
&=&qj_{\B}(\overline{\mu_{0}}(1)\mu_{-x}(a))q\\
&=&qj_{\B}(\mu_{0}(\alpha_{x}(a)))q\ \ [\overline{\mu_{0}}(1)\mu_{-x}(a)=\overline{\mu_{0}}(1)\mu_{0}(\alpha_{x}(a))=\mu_{0}(\alpha_{x}(a))]\\
&=&(j_{\B}\circ\mu_{0})(\alpha_{x}(a))=\pi_{A}(\alpha_{x}(a)).\\
\end{array}
\end{eqnarray*}
Therefore $\pi_{A}(\alpha_{x}(a))=W_{x}\pi_{A}(a)W_{x}^{*}$, and $W_{x}^{*}W_{x}\pi_{A}(a)=\pi_{A}(a)W_{x}^{*}W_{x}$ follows by some similar computation which we skip it here. So, the pair $(\pi_{A},W)$ is a partial-isometric
covariant representation of $(A,\Gamma^{+},\alpha)$ in $\mathcal{M}(q(\B\times_{\beta}\Gamma)q)$, and hence there is a nondegenerate homomorphism
$\Psi:=\pi_{A}\times W:(A\times_{\alpha}^{\piso}\Gamma^{+},i_{A},v)\rightarrow \mathcal{M}(q(\B\times_{\beta}\Gamma)q)$ such that $\Psi(i_{A}(a))=\pi_{A}(a)$ and $\Psi(v_{x})=W_{x}$ for all $a\in A$ and $x\in\Gamma^{+}$. We claim that $\Psi$ is an isomorphism
of $A\times_{\alpha}^{\piso}\Gamma^{+}$ onto $q(\B\times_{\beta}\Gamma)q$. To see that $\Psi$ is injective, we apply \cite[Theorem 4.8]{LR}. So, assume that $\pi_{A}(a)(q-W_{x}^{*}W_{x})=0$, where $a\in A$ and $x>0$ in $\Gamma^{+}$. We have
\begin{eqnarray}
\label{eq1}
\begin{array}{rcl}
\pi_{A}(a)(q-W_{x}^{*}W_{x})&=&j_{\B}(\mu_{0}(a))[q-qj_{\Gamma}(x)\overline{j_{\B}}(\overline{\mu_{0}}(1))j_{\Gamma}(x)^{*}q]\\
&=&j_{\B}(\mu_{0}(a))[q-q\overline{j_{\B}}(\overline{\beta}_{x}(\overline{\mu_{0}}(1)))q]\\
&=&j_{\B}(\mu_{0}(a))[q-q\overline{j_{\B}}(\overline{\mu_{x}}(1))\overline{j_{\B}}(\overline{\mu_{0}}(1))]\\
&=&j_{\B}(\mu_{0}(a))[q-q\overline{j_{\B}}(\overline{\mu_{x}}(1)\overline{\mu_{0}}(1))]\\
&=&j_{\B}(\mu_{0}(a))[q-q\overline{j_{\B}}(\overline{\mu_{x}}(\overline{\alpha}_{x}(1)))]\ \big[\overline{\mu_{x}}(1)\overline{\mu_{0}}(1)=\overline{\mu_{x}}(1)\overline{\mu_{x}}(\overline{\alpha_{x}}(1))\big]\\
&=&j_{\B}(\mu_{0}(a))q-j_{\B}(\mu_{0}(a))q\overline{j_{\B}}(\overline{\mu_{x}}(\overline{\alpha}_{x}(1)))\\
&=&j_{\B}(\mu_{0}(a))-j_{\B}(\mu_{0}(a)\overline{\mu_{x}}(\overline{\alpha}_{x}(1)))\\
&=&j_{\B}(\mu_{0}(a))-j_{\B}(\overline{\mu_{x}}(\alpha_{x}(a)\overline{\alpha}_{x}(1)))\ \big[\mu_{0}(a)\overline{\mu_{x}}(\overline{\alpha}_{x}(1))=\mu_{x}(\alpha_{x}(a))\overline{\mu_{x}}(\overline{\alpha}_{x}(1))\big]\\
&=&j_{\B}(\mu_{0}(a))-j_{\B}(\mu_{x}(\alpha_{x}(a)))=j_{\B}(\mu_{0}(a)-\mu_{x}(\alpha_{x}(a))).\\
\end{array}
\end{eqnarray}
Therefore $\pi_{A}(a)(q-W_{x}^{*}W_{x})=0$ implies that $j_{\B}(\mu_{0}(a)-\mu_{x}(\alpha_{x}(a)))=0$. Now if $\rho\times U$ is the nondegenerate representation of $\B\times_{\beta}\Gamma$ corresponding to the covariant pair $(\rho,U)$ of
$(\B,\Gamma,\beta)$ in $\L(\ell^{2}(\Gamma,A))$ (see \S\ref{sec:alg B}), then
$j_{\B}(\mu_{0}(a)-\mu_{x}(\alpha_{x}(a)))=0$ gives us
$$\rho\times U[j_{\B}(\mu_{0}(a)-\mu_{x}(\alpha_{x}(a)))]=\rho(\mu_{0}(a)-\mu_{x}(\alpha_{x}(a)))=0$$
in $\L(\ell^{2}(\Gamma,A))$. So it follows that, for $\varepsilon_{0}\otimes a^{*}\in\ell^{2}(\Gamma)\otimes A\simeq\ell^{2}(\Gamma,A)$, we must have $$\rho(\mu_{0}(a)-\mu_{x}(\alpha_{x}(a)))(\varepsilon_{0}\otimes a^{*})=0.$$ But since $x>0$,
we get $\varepsilon_{0}\otimes aa^{*}=0$. So $aa^{*}=0$ which implies that $a=0$. Therefore by \cite[Theorem 4.8]{LR} $\Psi$ is injective.

To see that $\Psi(A\times_{\alpha}^{\piso}\Gamma^{+})=q(\B\times_{\beta}\Gamma)q$, first note that for $a\in A$ and $x,y\in\Gamma^{+}$, we have
\begin{eqnarray*}
\begin{array}{rcl}
\Psi(v_{x}^{*}i_{A}(a)v_{y})&=&W_{x}^{*}\pi_{A}(a)W_{y}\\
&=&qj_{\Gamma}(x)(j_{\B}\circ\mu_{0})(a)j_{\Gamma}(y)^{*}q\\
&=&qj_{\Gamma}(x)j_{\B}(\mu_{0}(a))j_{\Gamma}(y)^{*}q,
\end{array}
\end{eqnarray*}
which belongs to $q(\B\times_{\beta}\Gamma)q$. Thus $\Psi(A\times_{\alpha}^{\piso}\Gamma^{+})\subset q(\B\times_{\beta}\Gamma)q$. To see the other inclusion, as $\B\times_{\beta}\Gamma$ is spanned by elements of the form $j_{\B}(\mu_{r}(a))j_{\Gamma}(s)$,
where $a\in A$ and $r,s\in\Gamma$, it is enough to show that $\Psi(A\times_{\alpha}^{\piso}\Gamma^{+})$ contains each spanning element $qj_{\B}(\mu_{r}(a))j_{\Gamma}(s)q$ of $q(\B\times_{\beta}\Gamma)q$.
Without loss of generality, we can assume that $r\geq0$. So, if $0\leq r<s$, then by the covariance equation of $(j_{\B},j_{\Gamma})$, we get
\begin{eqnarray*}
\begin{array}{rcl}
qj_{\B}(\mu_{r}(a))j_{\Gamma}(s)q&=&qj_{\Gamma}(s)j_{\B}(\beta_{-s}(\mu_{r}(a)))q\\
&=&qj_{\Gamma}(s)j_{\B}(\mu_{r-s}(a))\overline{j_{\B}}(\overline{\mu_{0}}(1))q\\
&=&qj_{\Gamma}(s)j_{\B}(\mu_{r-s}(a)\overline{\mu_{0}}(1))q\ \ (r-s<0)\\
&=&qj_{\Gamma}(s)j_{\B}(\mu_{0}(\alpha_{s-r}(a))q\ \ (s-r>0)\\
&=&qj_{\Gamma}(s)(j_{\B}\circ\mu_{0})(\alpha_{s-r}(a))j_{\Gamma}(0)^{*}q\\
&=&W_{s}^{*}\pi_{A}(\alpha_{s-r}(a))W_{0}\\
&=&\Psi(v_{s}^{*}i_{A}(\alpha_{s-r}(a))v_{0})\in\Psi(A\times_{\alpha}^{\piso}\Gamma^{+}).\\
\end{array}
\end{eqnarray*}
If $r\geq s$, then
\begin{eqnarray*}
\begin{array}{rcl}
qj_{\B}(\mu_{r}(a))j_{\Gamma}(s)q&=&qj_{\B}(\beta_{r}(\mu_{0}(a)))j_{\Gamma}(s)q\\
&=&qj_{\Gamma}(r)j_{\B}(\mu_{0}(a))j_{\Gamma}(r)^{*}j_{\Gamma}(s)q\\
&=&qj_{\Gamma}(r)j_{\B}(\mu_{0}(a))j_{\Gamma}(r-s)^{*}j_{\Gamma}(s)^{*}j_{\Gamma}(s)q\ \ (r-s\geq 0)\\
&=&qj_{\Gamma}(r)(j_{\B}\circ\mu_{0})(a)j_{\Gamma}(r-s)^{*}q\\
&=&W_{r}^{*}\pi_{A}(a)W_{r-s}\\
&=&\Psi(v_{r}^{*}i_{A}(a)v_{r-s})\in\Psi(A\times_{\alpha}^{\piso}\Gamma^{+}).\\
\end{array}
\end{eqnarray*}
Thus it follows that $q(\B\times_{\beta}\Gamma)q\subset \Psi(A\times_{\alpha}^{\piso}\Gamma^{+})$, and therefore $\Psi(A\times_{\alpha}^{\piso}\Gamma^{+})=q(\B\times_{\beta}\Gamma)q$. Consequently $\Psi$ is an isomorphism of $A\times_{\alpha}^{\piso}\Gamma^{+}$ onto $q(\B\times_{\beta}\Gamma)q$.

Finally, to see that $A\times_{\alpha}^{\piso}\Gamma^{+}$ is a full corner in $\B\times_{\beta}\Gamma$, we have to show that $(\B\times_{\beta}\Gamma)q(\B\times_{\beta}\Gamma)$ is dense in $\B\times_{\beta}\Gamma$. If $\{a_{i}\}$ is an approximate identity in $A$, then for any spanning element
$j_{\B}(\mu_{r}(a))j_{\Gamma}(s)$ of $\B\times_{\beta}\Gamma$, we have $j_{\B}(\mu_{r}(a_{i}a))j_{\Gamma}(s)\rightarrow j_{\B}(\mu_{r}(a))j_{\Gamma}(s)$ in the norm topology of $\B\times_{\beta}\Gamma$. But
\begin{eqnarray*}
\begin{array}{rcl}
j_{\B}(\mu_{r}(a_{i}a))j_{\Gamma}(s)&=&j_{\B}(\beta_{r}(\mu_{0}(a_{i}a)))j_{\Gamma}(s)\\
&=&j_{\Gamma}(r)j_{\B}(\mu_{0}(a_{i}a))j_{\Gamma}(r)^{*}j_{\Gamma}(s)\\
&=&j_{\Gamma}(r)j_{\B}(\mu_{0}(a_{i})\overline{\mu_{0}}(1)\mu_{0}(a))j_{\Gamma}(-r)j_{\Gamma}(s)\\
&=&j_{\Gamma}(r)j_{\B}(\mu_{0}(a_{i}))\overline{j_{\B}}(\overline{\mu_{0}}(1))j_{\B}(\mu_{0}(a))j_{\Gamma}(s-r)\\
&=&[j_{\Gamma}(r)j_{\B}(\mu_{0}(a_{i}))]q[j_{\B}(\mu_{0}(a))j_{\Gamma}(s-r)].\\
\end{array}
\end{eqnarray*}
The bottom line belongs to $(\B\times_{\beta}\Gamma)q(\B\times_{\beta}\Gamma)$, so does $j_{\B}(\mu_{r}(a_{i}a))j_{\Gamma}(s)$. Therefore we must have $j_{\B}(\mu_{r}(a))j_{\Gamma}(s)\in\overline{(\B\times_{\beta}\Gamma)q(\B\times_{\beta}\Gamma)}$, which implies that $\overline{(\B\times_{\beta}\Gamma)q(\B\times_{\beta}\Gamma)}=\B\times_{\beta}\Gamma$. So $A\times_{\alpha}^{\piso}\Gamma^{+}$ is a full corner in $\B\times_{\beta}\Gamma$. This completes the proof.
\end{proof}

\begin{lemma}
\label{ker}
The ideal $J$ of $A\times_{\alpha}^{\piso}\Gamma^{+}$ is isomorphic to $q(\I\times_{\beta}\Gamma)q$, which is a full corner in $\I\times_{\beta}\Gamma$.
\end{lemma}

\begin{proof}
We show that the isomorphism $\Psi$ in Theorem \ref{main} maps the ideal $J$ onto $q(\I\times_{\beta}\Gamma)q$. But first note that as $\I\times_{\beta}\Gamma$ sits in $\B\times_{\beta}\Gamma$ as an essential ideal, $\B\times_{\beta}\Gamma$ and
therefore $\M(\B\times_{\beta}\Gamma)$ are embedded in $\M(\I\times_{\beta}\Gamma)$ as $C^{*}$-subalgebras. Thus $q\in \M(\I\times_{\beta}\Gamma)$. Also $\I\times_{\beta}\Gamma$ is spanned by elements of the form
$j_{\B}(\mu_{r}(a)-\mu_{s}(\alpha_{s-r}(a)))j_{\Gamma}(z)$, where $a\in A$ and $r,s,z\in\Gamma$ such that $r<s$. Now to see that $\Psi(J)=q(\I\times_{\beta}\Gamma)q$, if $a\in A$ and $x,y,t\in\Gamma^{+}$, then
$$\Psi(v_{x}^{*}i_{A}(a)(1-v_{t}^{*}v_{t})v_{y})=W_{x}^{*}\pi_{A}(a)(q-W_{t}^{*}W_{t})W_{y}.$$ By the same calculation as (\ref{eq1}), we have $\pi_{A}(a)(q-W_{t}^{*}W_{t})=j_{\B}(\mu_{0}(a)-\mu_{t}(\alpha_{t}(a)))$, and therefore we obtain
\begin{align}
\label{eq2}
W_{x}^{*}\pi_{A}(a)(q-W_{t}^{*}W_{t})W_{y}=q[j_{\Gamma}(x)j_{\B}(\mu_{0}(a)-\mu_{t}(\alpha_{t}(a)))j_{\Gamma}(y)^{*}]q,
\end{align}
which belongs to $q(\I\times_{\beta}\Gamma)q$. So it follows that $\Psi(J)\subset q(\I\times_{\beta}\Gamma)q$. For the other inclusion, we show that each spanning element
$q[j_{\B}(\mu_{r}(a)-\mu_{s}(\alpha_{s-r}(a)))j_{\Gamma}(z)]q$ of $q(\I\times_{\beta}\Gamma)q$ belongs to $\Psi(J)$. To do this, we can assume that $0\leq r<s$ without loss of generality. So, if $0\leq r<s\leq z$,
then by the covariance equation of $(j_{\B},j_{\Gamma})$, we have
\begin{eqnarray*}
\begin{array}{l}
q[j_{\B}(\mu_{r}(a)-\mu_{s}(\alpha_{s-r}(a)))j_{\Gamma}(z)]q\\
=q[j_{\Gamma}(z)(j_{\B}\circ\beta_{-z})\big(\mu_{r}(a)-\mu_{s}(\alpha_{s-r}(a))\big)]q\\
=q[j_{\Gamma}(z)j_{\B}\big(\mu_{r-z}(a)-\mu_{s-z}(\alpha_{s-r}(a))\big)\overline{j_{\B}}(\overline{\mu_{0}}(1))]q\\
=q[j_{\Gamma}(z)j_{\B}\big(\mu_{r-z}(a)\overline{\mu_{0}}(1)-\mu_{s-z}(\alpha_{s-r}(a))\overline{\mu_{0}}(1)\big)]q\\
=q[j_{\Gamma}(z)j_{\B}\big(\mu_{0}(\alpha_{z-r}(a))-\mu_{0}(\alpha_{z-s}(\alpha_{s-r}(a)))\big)]q\ \ (\textrm{because}\ r-z,s-z\leq0)\\
=q[j_{\Gamma}(z)j_{\B}\big(\mu_{0}(\alpha_{z-r}(a))-\mu_{0}(\alpha_{z-r}(a))\big)]q=0\in\Psi(J).
\end{array}
\end{eqnarray*}
If $0\leq r<z<s$, then
\begin{eqnarray*}
\begin{array}{l}
q[j_{\B}(\mu_{r}(a)-\mu_{s}(\alpha_{s-r}(a)))j_{\Gamma}(z)]q\\
=q[j_{\Gamma}(z)(j_{\B}\circ\beta_{-z})\big(\mu_{r}(a)-\mu_{s}(\alpha_{s-r}(a))\big)]q\\
=q[j_{\Gamma}(z)j_{\B}\big(\mu_{r-z}(a)-\mu_{s-z}(\alpha_{s-r}(a))\big)\overline{j_{\B}}(\overline{\mu_{0}}(1))]q\\
=q[j_{\Gamma}(z)j_{\B}\big(\mu_{r-z}(a)\overline{\mu_{0}}(1)-\mu_{s-z}(\alpha_{s-r}(a))\overline{\mu_{0}}(1)\big)]q\\
=q[j_{\Gamma}(z)j_{\B}\big(\mu_{0}(\alpha_{z-r}(a))-\mu_{s-z}(\alpha_{s-r}(a)\overline{\alpha}_{s-z}(1))\big)]q.\ \ (\textrm{because}\ r-z<0, s-z>0)\\
\end{array}
\end{eqnarray*}
In the bottom line, for $\alpha_{s-r}(a)\overline{\alpha}_{s-z}(1)$, since $z-r>0$, we have
$$\alpha_{s-r}(a)\overline{\alpha}_{s-z}(1)=\alpha_{s-z}(\alpha_{z-r}(a))\overline{\alpha}_{s-z}(1)=\alpha_{s-z}(\alpha_{z-r}(a)1)=\alpha_{s-r}(a).$$
Let $b=\alpha_{z-r}(a)$, so, as $\alpha_{s-r}(a)=\alpha_{s-z}(\alpha_{z-r}(a))=\alpha_{s-z}(b)$, we get
\begin{eqnarray*}
\begin{array}{l}
q[j_{\B}(\mu_{r}(a)-\mu_{s}(\alpha_{s-r}(a)))j_{\Gamma}(z)]q\\
=q[j_{\Gamma}(z)j_{\B}\big(\mu_{0}(\alpha_{z-r}(a))-\mu_{s-z}(\alpha_{s-r}(a))\big)]q\\
=q[j_{\Gamma}(z)j_{\B}\big(\mu_{0}(b)-\mu_{s-z}(\alpha_{s-z}(b))\big)]q.\\
\end{array}
\end{eqnarray*}
Now for convenience, let $t=s-z>0$. Then by applying the equation (\ref{eq2}), we obtain
\begin{eqnarray*}
\begin{array}{rcl}
q[j_{\B}(\mu_{r}(a)-\mu_{s}(\alpha_{s-r}(a)))j_{\Gamma}(z)]q&=&q[j_{\Gamma}(z)j_{\B}\big(\mu_{0}(b)-\mu_{t}(\alpha_{t}(b))\big)j_{\Gamma}(0)^{*}]q\\
&=&W_{z}^{*}\pi_{A}(b)(q-W_{t}^{*}W_{t})W_{0}\\
&=&\Psi(v_{z}^{*}i_{A}(b)(1-v_{t}^{*}v_{t})v_{0})\in\Psi(J).\\
\end{array}
\end{eqnarray*}
At last, if $z\leq r$ (and $0\leq r<s$), then
\begin{eqnarray*}
\begin{array}{l}
q[j_{\B}(\mu_{r}(a)-\mu_{s}(\alpha_{s-r}(a)))j_{\Gamma}(z)]q\\
=q[(j_{\B}\circ\beta_{r})\big(\mu_{0}(a)-\mu_{s-r}(\alpha_{s-r}(a))\big)j_{\Gamma}(z)]q\\
=q[j_{\Gamma}(r)j_{\B}\big(\mu_{0}(a)-\mu_{s-r}(\alpha_{s-r}(a))\big)j_{\Gamma}(r)^{*}j_{\Gamma}(z)]q\\
=q[j_{\Gamma}(r)j_{\B}\big(\mu_{0}(a)-\mu_{s-r}(\alpha_{s-r}(a))\big)j_{\Gamma}(r-z)^{*}j_{\Gamma}(z)^{*}j_{\Gamma}(z)]q\\
=q[j_{\Gamma}(r)j_{\B}\big(\mu_{0}(a)-\mu_{t}(\alpha_{t}(a))\big)j_{\Gamma}(r-z)^{*}]q\ \ (r-z\geq0),\\
\end{array}
\end{eqnarray*}
where $t=s-r>0$. Therefore again by applying the equation (\ref{eq2}), we have
\begin{eqnarray*}
\begin{array}{rcl}
q[j_{\B}(\mu_{r}(a)-\mu_{s}(\alpha_{s-r}(a)))j_{\Gamma}(z)]q&=&q[j_{\Gamma}(r)j_{\B}\big(\mu_{0}(a)-\mu_{t}(\alpha_{t}(a))\big)j_{\Gamma}(r-z)^{*}]q\\
&=&W_{r}^{*}\pi_{A}(a)(q-W_{t}^{*}W_{t})W_{r-z}\\
&=&\Psi(v_{r}^{*}i_{A}(a)(1-v_{t}^{*}v_{t})v_{r-z})\in\Psi(J).\\
\end{array}
\end{eqnarray*}
Thus $\Psi(J)=q(\I\times_{\beta}\Gamma)q$, which implies that the ideal $J$ is isomorphic to $q(\I\times_{\beta}\Gamma)q$ via the isomorphism $\Psi$.

To show that $J\simeq q(\I\times_{\beta}\Gamma)q$ is a full corner in $\I\times_{\beta}\Gamma$, take an approximate identity in $\{a_{i}\}$ in $A$. Then for any spanning element spanned $j_{\B}(\mu_{r}(a)-\mu_{s}(\alpha_{s-r}(a)))j_{\Gamma}(z)$ of $\I\times_{\beta}\Gamma$,
where $a\in A$ and $r,s,z\in\Gamma$ such that $r<s$,
$$j_{\B}(\mu_{r}(aa_{i})-\mu_{s}(\alpha_{s-r}(aa_{i})))j_{\Gamma}(z)\rightarrow j_{\B}(\mu_{r}(a)-\mu_{s}(\alpha_{s-r}(a)))j_{\Gamma}(z)$$ in the norm topology of $\I\times_{\beta}\Gamma$.
But for $j_{\B}(\mu_{r}(aa_{i})-\mu_{s}(\alpha_{s-r}(aa_{i})))j_{\Gamma}(z)$, first, by the covariance equation of $(j_{\B},j_{\Gamma})$, we have
\begin{eqnarray*}
\begin{array}{l}
j_{\B}(\mu_{r}(aa_{i})-\mu_{s}(\alpha_{s-r}(aa_{i})))j_{\Gamma}(z)\\
=(j_{\B}\circ\beta_{r})\big(\mu_{0}(aa_{i})-\mu_{s-r}(\alpha_{s-r}(aa_{i}))\big)j_{\Gamma}(z)\\
=j_{\Gamma}(r)j_{\B}\big(\mu_{0}(aa_{i})-\mu_{s-r}(\alpha_{s-r}(aa_{i}))\big)j_{\Gamma}(r)^{*}j_{\Gamma}(z).\\
\end{array}
\end{eqnarray*}
Then in the bottom line, for $\mu_{0}(aa_{i})-\mu_{s-r}(\alpha_{s-r}(aa_{i}))$, we apply the equation (\ref{eq3}) to get
\begin{eqnarray*}
\begin{array}{l}
j_{\Gamma}(r)j_{\B}\big(\mu_{0}(aa_{i})-\mu_{s-r}(\alpha_{s-r}(aa_{i}))\big)j_{\Gamma}(r)^{*}j_{\Gamma}(z)\\
=j_{\Gamma}(r)j_{\B}\big([\mu_{0}(a)-\mu_{s-r}(\alpha_{s-r}(a))][\mu_{0}(a_{i})-\mu_{s-r}(\alpha_{s-r}(a_{i}))]\big)j_{\Gamma}(-r)j_{\Gamma}(z)\\
=j_{\Gamma}(r)j_{\B}\big(\mu_{0}(a)-\mu_{s-r}(\alpha_{s-r}(a))\big)j_{\B}\big(\mu_{0}(a_{i})-\mu_{s-r}(\alpha_{s-r}(a_{i}))\big)j_{\Gamma}(z-r)\\
=j_{\Gamma}(r)j_{\B}\big(\mu_{0}(a)\overline{\mu_{0}}(1)-\mu_{s-r}(\alpha_{s-r}(a))\overline{\mu_{0}}(1)\big)j_{\B}\big(\overline{\mu_{0}}(1)\mu_{0}(a_{i})-\overline{\mu_{0}}(1)\mu_{s-r}(\alpha_{s-r}(a_{i}))\big)j_{\Gamma}(z-r)\\
=j_{\Gamma}(r)j_{\B}\big(\mu_{0}(a)-\mu_{s-r}(\alpha_{s-r}(a))\big)\overline{j_{\B}}(\overline{\mu_{0}}(1))j_{\B}\big(\mu_{0}(a_{i})-\mu_{s-r}(\alpha_{s-r}(a_{i}))\big)j_{\Gamma}(z-r)\\
=\big[j_{\Gamma}(r)j_{\B}\big(\mu_{0}(a)-\mu_{s-r}(\alpha_{s-r}(a))\big)\big]q\big[j_{\B}\big(\mu_{0}(a_{i})-\mu_{s-r}(\alpha_{s-r}(a_{i}))\big)j_{\Gamma}(z-r)\big],
\end{array}
\end{eqnarray*}
which belongs to $(\I\times_{\beta}\Gamma)q(\I\times_{\beta}\Gamma)$. So we must have $j_{\B}(\mu_{r}(a)-\mu_{s}(\alpha_{s-r}(a)))j_{\Gamma}(z)\in\overline{(\I\times_{\beta}\Gamma)q(\I\times_{\beta}\Gamma)}$. Therefore
$\overline{(\I\times_{\beta}\Gamma)q(\I\times_{\beta}\Gamma)}=\I\times_{\beta}\Gamma$, which means that $J\simeq q(\I\times_{\beta}\Gamma)q$ is a full corner in $\I\times_{\beta}\Gamma$.
\end{proof}

\begin{prop}
\label{essen}
The ideal $J$ is an essential ideal of $A\times_{\alpha}^{\piso}\Gamma^{+}$, and this generalizes \cite[Propostion 2.5]{AZ} to any totally ordered abelian (discrete) group $\Gamma$, not only subgroups of $\R$.
\end{prop}

\begin{proof}
One can see this by knowing that the ideal $J$ and $A\times_{\alpha}^{\piso}\Gamma^{+}$ are full corners in $\I\times_{\beta}\Gamma$ and $\B\times_{\beta}\Gamma$ respectively, and the fact that $\I\times_{\beta}\Gamma$ is essential ideal of
$\B\times_{\beta}\Gamma$.
\end{proof}

\section{The partial-isometric crossed products by semigroups of automorphisms}
\label{sec:auto}
In this section we suppose that $(A,\Gamma^{+},\alpha)$ is a system consisting of a $C^{*}$-algebra $A$ and an action $\alpha:\Gamma^{+}\rightarrow \Aut(A)$ of $\Gamma^{+}$ by automorphisms of $A$. If for every $s<0$ in $\Gamma$,
we let $\alpha_{s}:=\alpha_{-s}^{-1}$, then $\alpha$ extends to an action of $\Gamma$ by automorphisms of $A$. Now let $B_{\Gamma}$ be the subalgebra of $\ell^{\infty}(\Gamma)$ spanned by $\{1_{s}\in\ell^{\infty}(\Gamma):s\in\Gamma\}$, where
\[
1_{s}(t)=
   \begin{cases}
      1 &\textrm{if}\empty\ \text{$t\geq s$}\\
      0 &\textrm{if}\empty\ \text{$t<s$}.
   \end{cases}
\]
Then there is an action $\tau:\Gamma\rightarrow \Aut(B_{\Gamma})$ of $\Gamma$ by automorphisms of $B_{\Gamma}$ given by translation, such that $\tau_{x}(1_{s})=1_{s+x}$ for all $x,s\in\Gamma$. Moreover, $x\mapsto \tau_{x}\otimes \alpha_{x}^{-1}$
defines an action of $\Gamma$ on the algebra $B_{\Gamma}\otimes A$ by automorphisms, and therefore we obtain a classical dynamical system $(B_{\Gamma}\otimes A, \Gamma, \tau\otimes \alpha^{-1})$. Also note that
$$B_{\Gamma,\infty}:=\clsp\{1_{s}-1_{t}: s<t\in \Gamma\}$$ is an ideal of $B_{\Gamma}$ which is $\tau$-invariant. So it follows that $B_{\Gamma,\infty}\otimes A$ is a $(\tau\otimes \alpha^{-1})$-invariant ideal of $B_{\Gamma}\otimes A$.
We want to recover \cite[Corollary 5.3]{AZ} from our discussion here in \S\ref{sec:full piso}.

\begin{prop}
\label{B-auto}
There is an isomorphism $\delta:B_{\Gamma}\otimes A\rightarrow \B$ such that $\beta_{x}(\delta(\xi))=\delta((\tau\otimes\alpha^{-1})_{x}(\xi))$ for all $\xi\in(B_{\Gamma}\otimes A)$ and $x\in\Gamma$, and it maps the ideal $B_{\Gamma,\infty}\otimes A$ onto $\I$.
Then $\delta$ induces an isomorphism $\Delta$ of $((B_{\Gamma}\otimes A)\times_{\tau\otimes\alpha^{-1}}\Gamma,k)$ onto $(\B\times_{\beta}\Gamma,j)$ such that
\begin{align}
\label{eq5}
\Delta(k_{B_{\Gamma}\otimes A}(\xi)k_{\Gamma}(s))=j_{\B}(\delta(\xi))j_{\Gamma}(s)\ \ \textrm{for all}\ \xi\in(B_{\Gamma}\otimes A), s\in\Gamma,
\end{align}
and it maps the ideal $(B_{\Gamma,\infty}\otimes A)\times_{\tau\otimes\alpha^{-1}}\Gamma$ onto $\I\times_{\beta}\Gamma$.
\end{prop}

\begin{proof}
First, recall that as the algebra $B_{\Gamma}$ is abelian, $B_{\Gamma}\otimes A=B_{\Gamma}\otimes_{\textrm{max}} A=B_{\Gamma}\otimes_{\textrm{min}} A$. Also note that since $\ell^{\infty}(\Gamma,A)$ sits in $\ell^{\infty}(\Gamma,\M(A))$ as an essential ideal, $\ell^{\infty}(\Gamma,\M(A))$ is embedded in $\M(\ell^{\infty}(\Gamma,A))$ as a  $C^*$-subalgebra. Now define the maps
$$\varphi:B_{\Gamma}\rightarrow \M(\ell^{\infty}(\Gamma,A))\ \ \textrm{and}\ \ \psi:A\rightarrow \M(\ell^{\infty}(\Gamma,A))$$
by $$\varphi(f)(x)=f(x)1_{\M(A)}\ \ \textrm{and}\ \ \psi(a)(x)=\alpha_{x}(a)$$ for every $f\in B_{\Gamma}$, $a\in A$, and $x\in\Gamma$. One can see that $\varphi$ and $\psi$ are $*$-homomorphisms  with commuting ranges, meaning that
$\varphi(f)\psi(a)=\psi(a)\varphi(f)$ for all $f\in B_{\Gamma}$ and $a\in A$. Therefore there is a homomorphism $\delta:=\varphi\otimes\psi:B_{\Gamma}\otimes A\rightarrow \M(\ell^{\infty}(\Gamma,A))$ such that
$\delta(f\otimes a)=\varphi(f)\psi(a)=\psi(a)\varphi(f)$ for every $f\in B_{\Gamma}$ and $a\in A$. We claim that $\delta$ is an isomorphism of $B_{\Gamma}\otimes A$ onto the algebra $\B$. To see this, we first show that $\delta(B_{\Gamma}\otimes A)=\B$. For each spanning element
$1_{x}\otimes a$ of $B_{\Gamma}\otimes A$, we have
\[
\delta(1_{x}\otimes a)(y)=(\varphi(1_{x})\psi(a))(y)=\varphi(1_{x})(y)\psi(a)(y)=1_{x}(y)\alpha_{y}(a)=
   \begin{cases}
      \alpha_{y}(a) &\textrm{if}\empty\ \text{$y\geq x$}\\
      0 &\textrm{if}\empty\ \text{$y<x$},
   \end{cases}
\]
which is equal to
\[
\mu_{x}(\alpha_{x}(a))(y)=
   \begin{cases}
      \alpha_{y-x}(\alpha_{x}(a))=\alpha_{y}(a) &\textrm{if}\empty\ \text{$y\geq x$}\\
      0 &\textrm{if}\empty\ \text{$y<x$},
   \end{cases}
\]
for every $y\in\Gamma$. So it follows that $\delta(1_{x}\otimes a)=\mu_{x}(\alpha_{x}(a))\in\B$, and therefore $\delta(B_{\Gamma}\otimes A)\subset\B$. Conversely, for any spanning element $\mu_{x}(a)$ of $\B$, by applying
$\delta(1_{x}\otimes a)=\mu_{x}(\alpha_{x}(a))$, we get
$$\delta(1_{x}\otimes \alpha_{x}^{-1}(a))=\mu_{x}(\alpha_{x}(\alpha_{x}^{-1}(a)))=\mu_{x}(a).$$ Thus $\mu_{x}(a)=\delta(1_{x}\otimes \alpha_{x}^{-1}(a))\in\delta(B_{\Gamma}\otimes A)$, which implies that $\B\subset\delta(B_{\Gamma}\otimes A)$.

Now we show that $\delta$ is injective. Define the map $M:B_{\Gamma}\rightarrow B(\ell^{2}(\Gamma))$ by $(M(f)\lambda)(x)=f(x)\lambda(x)$ for every $f\in B_{\Gamma}$ and $\lambda\in\ell^{2}(\Gamma)$, which is a faithful (nondegenerate) representation. Then let
$\pi:A\rightarrow B(H)$ be a faithful (nondegenerate) representation of $A$ on some Hilbert space $H$. By \cite[Corollary B.11]{RW}, there is a faithful representation $M\otimes\pi:B_{\Gamma}\otimes A\rightarrow B(\ell^{2}(\Gamma)\otimes H)$ such that
$M\otimes\pi(f\otimes a)=M(f)\otimes \pi(a)$. On the other hand, the map $\tilde{\pi}:\B\rightarrow B(\ell^{2}(\Gamma,H))$ defined by $(\tilde{\pi}(\xi)\eta)(x)=\pi(\alpha_{x}^{-1}(\xi(x)))\eta(x)$, where $\xi\in\B$ and $\eta\in\ell^{2}(\Gamma,H)$, is a faithful representation of
$\B$ on the Hilbert space $\ell^{2}(\Gamma,H)$. Now let $U$ be the isomorphism (unitary) of $\ell^{2}(\Gamma)\otimes H$ onto $\ell^{2}(\Gamma,H)$ such that $U(\lambda\otimes h)(x)=\lambda(x)h$ for all $\lambda\in\ell^{2}(\Gamma)$ and $h\in H$. Then we have
\begin{eqnarray*}
\begin{array}{rcl}
\big(\tilde{\pi}(\delta(f\otimes a))U(\lambda\otimes h)\big)(x)&=&\pi(\alpha_{x}^{-1}(\delta(f\otimes a)(x)))U(\lambda\otimes h)(x)\\
&=&\pi(\alpha_{x}^{-1}(f(x)\alpha_{x}(a)))(\lambda(x)h)\\
&=&\pi(f(x)a)(\lambda(x)h)\\
&=&f(x)\lambda(x)\pi(a)(h),
\end{array}
\end{eqnarray*}
and
\begin{eqnarray*}
\begin{array}{rcl}
U\big((M\otimes\pi(f\otimes a))(\lambda\otimes h)\big)(x)&=&U\big((M(f)\otimes \pi(a))(\lambda\otimes h)\big)(x)\\
&=&U\big(M(f)\lambda\otimes \pi(a)(h))(x)\\
&=&(M(f)\lambda)(x)\pi(a)(h)=f(x)\lambda(x)\pi(a)(h).
\end{array}
\end{eqnarray*}
Thus $\tilde{\pi}(\delta(f\otimes a))U(\lambda\otimes h)=U\big((M\otimes\pi(f\otimes a))(\lambda\otimes h)\big)$. It therefore follows that
$$U^{*}\tilde{\pi}(\delta(\xi))U=(M\otimes\pi)(\xi)$$ for all $\xi\in B_{\Gamma}\otimes A$, from which we conclude that $\delta$ must be injective. This is because $\tilde{\pi}$ and $M\otimes\pi$ are injective, and $U$ is a unitary.
Consequently $B_{\Gamma}\otimes A\simeq\delta(B_{\Gamma}\otimes A)=\B$. Moreover, $B_{\Gamma,\infty}\otimes A$ is isomorphic to
$\I$ via $\delta$. This is because for $a\in A$ and $x<y\in\Gamma$,
\begin{eqnarray}
\label{eq4}
\begin{array}{rcl}
\delta((1_{x}-1_{y})\otimes a)&=&\delta((1_{x}\otimes a)-(1_{y}\otimes a))\\
&=&\delta(1_{x}\otimes a)-\delta(1_{y}\otimes a)\\
&=&\mu_{x}(\alpha_{x}(a))-\mu_{y}(\alpha_{y}(a))\\
&=&\mu_{x}(\alpha_{x}(a))-\mu_{y}(\alpha_{y-x}(\alpha_{x}(a)))\in\I.
\end{array}
\end{eqnarray}
Thus $\delta(B_{\Gamma,\infty}\otimes A)\subset\I$. For the other inclusion, by the computation in (\ref{eq4}),
\begin{eqnarray*}
\begin{array}{rcl}
\delta((1_{x}-1_{y})\otimes \alpha_{x}^{-1}(a))&=&\delta((1_{x}-1_{y})\otimes \alpha_{-x}(a))\\
&=&\mu_{x}(\alpha_{x}(\alpha_{-x}(a)))-\mu_{y}(\alpha_{y}(\alpha_{-x}(a)))\\
&=&\mu_{x}(a)-\mu_{y}(\alpha_{y-x}(a)).
\end{array}
\end{eqnarray*}
Therefore each spanning element $\mu_{x}(a)-\mu_{y}(\alpha_{y-x}(a))$ of $\I$ equals $\delta((1_{x}-1_{y})\otimes \alpha_{x}^{-1}(a))$, which belongs to $\delta(B_{\Gamma,\infty}\otimes A)$. So $\I\subset\delta(B_{\Gamma,\infty}\otimes A)$, and therefore
$\delta(B_{\Gamma,\infty}\otimes A)=\I$. This implies that $B_{\Gamma,\infty}\otimes A\simeq\I$ via $\delta$.

Finally we show that the isomorphism $\delta$ satisfies $\beta_{x}\circ\delta=\delta\circ(\tau\otimes\alpha^{-1})_{x}$. Therefore by \cite[Lemma 2.65]{W}, there is an isomorphism
$\Delta:((B_{\Gamma}\otimes A)\times_{\tau\otimes\alpha^{-1}}\Gamma,k)\rightarrow (\B\times_{\beta}\Gamma,j)$ such that
$$\Delta(k_{B_{\Gamma}\otimes A}(\xi)k_{\Gamma}(s))=j_{\B}(\delta(\xi))j_{\Gamma}(s)\ \ \textrm{for all}\ \xi\in(B_{\Gamma}\otimes A), s\in\Gamma.$$
For any spanning element $1_{s}\otimes a$ of $B_{\Gamma}\otimes A$, we have
$$\beta_{x}(\delta(1_{s}\otimes a))=\beta_{x}(\mu_{s}(\alpha_{s}(a)))=\mu_{s+x}(\alpha_{s}(a)).$$
On the other hand,
\begin{eqnarray*}
\begin{array}{rcl}
\delta((\tau\otimes\alpha^{-1})_{x}(1_{s}\otimes a))&=&\delta(\tau_{x}\otimes\alpha^{-1}_{x}(1_{s}\otimes a))\\
&=&\delta(\tau_{x}(1_{s})\otimes\alpha^{-1}_{x}(a))\\
&=&\delta(1_{s+x}\otimes\alpha_{-x}(a))\\
&=&\mu_{s+x}(\alpha_{s+x}(\alpha_{-x}(a)))\ \ [\textrm{by applying}\ \delta(1_{t}\otimes a)=\mu_{t}(\alpha_{t}(a))]\\
&=&\mu_{s+x}(\alpha_{s}(a)).\\
\end{array}
\end{eqnarray*}
So $\beta_{x}\circ\delta=\delta\circ(\tau\otimes\alpha^{-1})_{x}$ is valid. Note that
$$(B_{\Gamma,\infty}\otimes A)\times_{\tau\otimes\alpha^{-1}}\Gamma\simeq\Delta\big((B_{\Gamma,\infty}\otimes A)\times_{\tau\otimes\alpha^{-1}}\Gamma\big)=\I\times_{\beta}\Gamma$$
follows by some routine computation on spanning elements using the equation (\ref{eq5}). We skip it here.
\end{proof}

\begin{cor}
\cite[Corollary 5.3]{AZ}\label{full piso-auto}
Let $p=\overline{k}_{B_{\Gamma}\otimes A}(1_{0}\otimes 1_{\M(A)})\in\M\big((B_{\Gamma}\otimes A)\times_{\tau\otimes\alpha^{-1}}\Gamma\big)$. We have $\overline{\Delta}(p)=q$, and therefore $A\times_{\alpha}^{\piso}\Gamma^{+}$ and the ideal $J$ are isomorphic to the full corners $p[(B_{\Gamma}\otimes A)\times_{\tau\otimes\alpha^{-1}}\Gamma]p$ and $p[(B_{\Gamma,\infty}\otimes A)\times_{\tau\otimes\alpha^{-1}}\Gamma]p$ respectively.
\end{cor}

\begin{proof}
First note that since the homomorphism $j_{\B}$ is nondegenerate, so is the isomorphism $\Delta$. Therefore $\Delta$ extends to an isometry of multiplier algebras. Now let $\{a_{i}\}$ be an approximate identity in $A$. By using the equation (\ref{eq5}), we have
$$\Delta(k_{B_{\Gamma}\otimes A}(1_{0}\otimes a_{i}))=j_{\B}(\delta(1_{0}\otimes a_{i}))=j_{\B}(\mu_{0}(a_{i})).$$
So, as $1_{0}\otimes a_{i}\rightarrow 1_{0}\otimes 1_{\M(A)}$ in $\M(B_{\Gamma}\otimes A)$ strictly, in the equation above, the left hand side approaches $\overline{\Delta}(\overline{k}_{B_{\Gamma}\otimes A}(1_{0}\otimes 1_{\M(A)}))=\overline{\Delta}(p)$,
while the right hand side approaches $\overline{j_{\B}}(\overline{\mu_{0}}(1))=q$ in $\M(\B\times_{\beta}\Gamma)$ strictly. Thus we must have $\overline{\Delta}(p)=q$. Therefore, by Proposition \ref{B-auto}, it follows that
$$p[(B_{\Gamma}\otimes A)\times_{\tau\otimes\alpha^{-1}}\Gamma]p\simeq\Delta\big(p[(B_{\Gamma}\otimes A)\times_{\tau\otimes\alpha^{-1}}\Gamma]p\big)=q(\B\times_{\beta}\Gamma)q,$$ and
$$p[(B_{\Gamma,\infty}\otimes A)\times_{\tau\otimes\alpha^{-1}}\Gamma]p\simeq\Delta\big(p[(B_{\Gamma,\infty}\otimes A)\times_{\tau\otimes\alpha^{-1}}\Gamma]p\big)=q(\I\times_{\beta}\Gamma)q,$$
which by Theorem \ref{main} and Lemma \ref{ker}, are isomorphic to $A\times_{\alpha}^{\piso}\Gamma^{+}$ and the ideal $J$ respectively. Consequently,
$$A\times_{\alpha}^{\piso}\Gamma^{+}\simeq p[(B_{\Gamma}\otimes A)\times_{\tau\otimes\alpha^{-1}}\Gamma]p\ \ \textrm{and}\ \ J\simeq p[(B_{\Gamma,\infty}\otimes A)\times_{\tau\otimes\alpha^{-1}}\Gamma]p$$ via the isomorphism $\Delta^{-1}\circ \Psi$.
\end{proof}

\section{The partial-isometric crossed products by a single endomorphism}
\label{sec:integer}
In this section we want to show that when $\Gamma=\Z$, \cite[Theorem 4.1]{AZ} follows by our discussion here in \S\ref{sec:full piso}. First of all, if $\Gamma=\Z$, then the algebra $\B$ associated to the system $(A,\N,\alpha)$ (generated by the single endomorphism $\alpha:=\alpha_{1}$) is the subalgebra $D$ of $\ell^{\infty}(\Z,A)$ in \cite[\S 1]{KS}. So it is precisely the subalgebra of all elements $\xi\in \ell^{\infty}(\Z,A)$ such that $\lim_{n\rightarrow -\infty} \|\xi(n)\|=0$, and for each $\varepsilon>0$, there exists $m\in \Z$ such that $\|\xi(n)-\alpha_{n-m}(\xi(m))\|<\varepsilon$
for every $n\geq m$ (see Lemma \ref{B property} and \cite[Lemma 1.7]{KS}). Also $\I$ is the ideal $C_{0}(\Z,A)$. This is because for every $a\in A$ and $n\in\Z$, $$(...,0,0,0,a,0,0,0,...)=\mu_{n}(a)-\mu_{n+1}(\alpha(a)),$$ where $a$ is the $n$th slot. Thus $C_{0}(\Z,A)\subset \I$. Conversely, for every $a\in A$ and $m<n\in \Z$,
$$\mu_{m}(a)-\mu_{n}(\alpha_{n-m}(a))=(...,0,0,0,a,\alpha(a),\alpha_{2}(a),...,\alpha_{n-m-1}(a),0,0,0,...),$$ which obviously belongs to $C_{0}(\Z,A)$. This implies that $\I\subset C_{0}(\Z,A)$, and therefore $\I=C_{0}(\Z,A)$. Moreover, in this case, the kernel of the homomorphism $\sigma$
in Lemma \ref{B extension} is precisely $\I=C_{0}(\Z,A)$ (see also \cite[Lemma 1.7]{KS}).

Now, by Theorem \ref{main}, the partial-isometric crossed product $A\times_{\alpha}^{\piso}\N$ of $(A,\N,\alpha)$ is the full corner $q(\B\times_{\beta}\Z)q$, and the ideal $J$ is the full corner $q(C_{0}(\Z,A)\times_{\beta}\Z)q$ by Lemma \ref{ker}. Furthermore, the nondegenerate representation of $\rho\times U$ of $\B\times_{\beta}\Z$ associated to the covariant pair $(\rho,U)$ of $(\B,\Z,\beta)$ in $\L(\ell^{2}(\Z,A))$ (see \S\ref{sec:alg B}), is faithful, which identifies $C_{0}(\Z,A)\times_{\beta}\Z$ with the algebra $\K(\ell^{2}(\Z,A))$ of compact operators on
$\ell^{2}(\Z,A)$ (see \cite[Theorem 1.8]{KS}).
So it follows that $(\rho\times U)\circ \Psi$ maps the ideal $J$ isomorphically on the full corner $\tilde{q}\K(\ell^{2}(\Z,A))\tilde{q}$ of $\K(\ell^{2}(\Z,A))$, where $\tilde{q}\in\M\big(\K(\ell^{2}(\Z,A))\big)=\L(\ell^{2}(\Z,A))$ is the projection $\overline{\rho\circ\mu_{0}}(1)$ such that
\[
(\tilde{q}(f))(n)=
   \begin{cases}
      \overline{\alpha}_{n}(1)f(n) &\textrm{if}\empty\ \text{$n\geq 0$}\\
      0 &\textrm{if}\empty\ \text{$n<0$},
   \end{cases}
\]
for all $f\in\ell^{2}(\Z,A)$. But $\tilde{q}\K(\ell^{2}(\Z,A))\tilde{q}$ is indeed the full corner $p\K(\ell^{2}(\N,A))p$ of $\K(\ell^{2}(\N,A))$, where $p\in\M\big(\K(\ell^{2}(\N,A))\big)=\L(\ell^{2}(\N,A))$ is the projection in \cite[Theorem 4.1]{AZ} given by
$(p(g))(m)=\overline{\alpha}_{m}(1)g(m)$ for every $g\in\ell^{2}(\N,A)$ and $m\in\N$. To see this, let $\{e_{n}\}_{n\in\Z}$ be the usual orthonormal basis for $\ell^{2}(\Z)$ ($\{e_{n}\}_{n\in\N}$ is the one for $\ell^{2}(\N)$ accordingly).
As $\K(\ell^{2}(\Z,A))$ is spanned by elements (compact operators) $\{\Theta_{e_{n}\otimes a, e_{m}\otimes b}: a,b\in A, m,n\in\Z\}$, we have
\begin{eqnarray*}
\begin{array}{rcl}
\tilde{q}\K(\ell^{2}(\Z,A))\tilde{q}&=&\clsp\{\tilde{q}(\Theta_{e_{n}\otimes a, e_{m}\otimes b})\tilde{q}: a,b\in A, m,n\in\Z\}\\
&=&\clsp\{\Theta_{\tilde{q}(e_{n}\otimes a), \tilde{q}(e_{m}\otimes b)}: a,b\in A, m,n\in\Z\}.
\end{array}
\end{eqnarray*}
But if $n<0$ or $m<0$, then $\tilde{q}(e_{n}\otimes a)=0$ or $\tilde{q}(e_{m}\otimes b)=0$, and therefore $\Theta_{\tilde{q}(e_{n}\otimes a), \tilde{q}(e_{m}\otimes b)}=0$. So it follows that
$$\tilde{q}\K(\ell^{2}(\Z,A))\tilde{q}=\clsp\{\Theta_{\tilde{q}(e_{n}\otimes a), \tilde{q}(e_{m}\otimes b)}: a,b\in A, m,n\in\N\}.$$ Now, since $\tilde{q}(e_{n}\otimes a)=p(e_{n}\otimes a)$ for every $a\in A$ and $n\in\N$, we get
\begin{eqnarray*}
\begin{array}{rcl}
\tilde{q}\K(\ell^{2}(\Z,A))\tilde{q}&=&\clsp\{\Theta_{p(e_{n}\otimes a), p(e_{m}\otimes b)}: a,b\in A, m,n\in\N\}\\
&=&\clsp\{p(\Theta_{e_{n}\otimes a, e_{m}\otimes b})p: a,b\in A, m,n\in\N\}=p\K(\ell^{2}(\N,A))p.\\
\end{array}
\end{eqnarray*}
So this is actually a recovery of \cite[Theorem 4.1]{AZ}.

\begin{lemma}
\label{direct lim}
Let $\alpha$ be a single injective (extendible) endomorphism of $A$. Then the algebra $\B$ associated to the system $(A,\N,\alpha)$ is the direct limit of the direct system $(\B_{n},\varphi_{n})_{n\in\Z}$, where for each $n$, $\B_{n}=B_{\Z}\otimes A$, and
$\varphi_{n}:\B_{n}\rightarrow \B_{n+1}$ is defined by
\[
(\varphi_{n}(\xi))(i)=
   \begin{cases}
      \alpha\big(\xi(i)\big) &\textrm{if}\empty\ \text{$i>n$}\\
      \xi(i) &\textrm{if}\empty\ \text{$i\leq n$},
   \end{cases}
\]
for every $\xi\in \B_{n}$.
\end{lemma}

\begin{proof}
For every $n\in \Z$, define the map $\psi^{n}:\B_{n}\rightarrow \B$ by
\[
(\psi^{n}(\xi))(i)=
   \begin{cases}
      \alpha_{i-n}\big(\xi(i)\big) &\textrm{if}\empty\ \text{$i\geq n$}\\
      \xi(i) &\textrm{if}\empty\ \text{$i< n$},
   \end{cases}
\]
for every $\xi\in \B_{n}$. We only verify that each $\psi^{n}(\xi)$ belongs to $\B$. Then it is not difficult to see that each $\psi^{n}$ is indeed an embedding (injective $*$-homomorphism) of $\B_{n}$ in $\B$, as $\alpha$ is injective.
Firstly, by viewing $\B_{n}=B_{\Z}\otimes A$ as the subalgebra of $\ell^{\infty}(\Z,A)$ consisting of all elements $\xi$ such that $\lim_{i\rightarrow -\infty} \|\xi(i)\|=0$, and $\lim_{i\rightarrow \infty} \|\xi(i)\|$ exists, for every
$\varepsilon>0$, there exists $m\in \Z$ such that for all $i\geq j>m$, $\|\xi(i)-\xi(j)\|<\varepsilon$. So it follows that $\lim_{i\rightarrow -\infty} \|(\psi^{n}(\xi))(i)\|=\lim_{i\rightarrow -\infty}\|\xi(i)\|=0$,
as $(\psi^{n}(\xi))(i)=\xi(i)$ for every $i< n$. Moreover if $k=\textrm{max}\{n,m\}+1$, then for every $i\geq k$,
\begin{eqnarray*}
\begin{array}{rcl}
\|(\psi^{n}(\xi))(i)-\alpha_{i-k}\big((\psi^{n}(\xi))(k)\big)\|&=&\|\alpha_{i-n}\big(\xi(i)\big)-\alpha_{i-k}\big(\alpha_{k-n}(\xi(k))\big)\|\\
&=&\|\alpha_{i-n}(\xi(i))-\alpha_{i-n}(\xi(k))\|\\
&=&\|\alpha_{i-n}(\xi(i)-\xi(k))\|\\
&=&\|\xi(i)-\xi(k)\|< \varepsilon,
\end{array}
\end{eqnarray*}
as $i\geq k> m$. Thus $\psi^{n}(\xi)\in \B$. Next we show that for every $n\in \Z$, $\psi^{n+1}\circ \varphi_{n}=\psi^{n}$. Therefore if $\B_{\infty}$ is the direct limit of the direct system $(\B_{n},\varphi_{n})_{n\in\Z}$, then by
\cite[Proposition 11.4.1(ii)]{KR}, there is a (unique) $*$-homomorphism $\psi: \B_{\infty}\rightarrow \B$ such that $\psi\circ \varphi^{n}=\psi^{n}$, where each $\varphi^{n}$ is the canonical homomorphism (embedding) of $\B_{n}$ into $\B_{\infty}$.
For ever $\xi\in \B_{n}$, we have
\[
\psi^{n+1}\big(\varphi_{n}(\xi)\big)(i)=
   \begin{cases}
      \alpha_{i-(n+1)}\big(\varphi_{n}(\xi)(i)\big) &\textrm{if}\empty\ \text{$i\geq n+1$}\\
      \varphi_{n}(\xi)(i) &\textrm{if}\empty\ \text{$i< n+1$}.
   \end{cases}
\]
So, for $i\geq n+1$, $$\psi^{n+1}\big(\varphi_{n}(\xi)\big)(i)=\alpha_{i-(n+1)}\big(\varphi_{n}(\xi)(i)\big)=\alpha_{i-n-1}\big(\alpha(\xi(i))\big)=\alpha_{i-n}\big(\xi(i)\big)=(\psi^{n}(\xi))(i).$$ If $i< n+1$, then
$$\psi^{n+1}\big(\varphi_{n}(\xi)\big)(i)=\varphi_{n}(\xi)(i)=\xi(i)=(\psi^{n}(\xi))(i).$$ Thus we have $\psi^{n+1}\circ \varphi_{n}=\psi^{n}$. We claim that $\psi$ is actually an isomorphism
of $\B_{\infty}$ onto $\B$. Since $\alpha$ is injective,
$$\|\psi(\varphi^{n}(\xi))\|=\|\psi^{n}(\xi)\|=\|\xi\|=\|\varphi^{n}(\xi)\|.$$ Therefore, since $\cup_{n\in \Z}\varphi^{n}(\B_{n})$ is a dense subalgebra of $\B_{\infty}$, $\psi$ is an isometry. Finally, $\psi$ is onto, as one can see that
$$\psi(\varphi^{n}(1_{n}\otimes a))=\psi^{n}(1_{n}\otimes a)=\mu_{n}(a),$$ which is a spanning element of $\B$. Thus $\psi$ is indeed an isomorphism of $\B_{\infty}$ onto $\B$.
 \end{proof}


\end{document}